\documentclass[11pt]{article}

\usepackage{amssymb,amsmath,amsfonts}
\usepackage{graphicx,color,enumitem}
\usepackage{calrsfs}	
\usepackage{amsthm}	
\usepackage{bm}
\usepackage[round]{natbib}
\usepackage{hyperref}

\usepackage{tikz}

\usepackage{etoolbox}

\makeatletter

\patchcmd{\NAT@citex}
  {\@citea\NAT@hyper@{%
     \NAT@nmfmt{\NAT@nm}%
     \hyper@natlinkbreak{\NAT@aysep\NAT@spacechar}{\@citeb\@extra@b@citeb}%
     \NAT@date}}
  {\@citea\NAT@nmfmt{\NAT@nm}%
   \NAT@aysep\NAT@spacechar\NAT@hyper@{\NAT@date}}{}{}

\patchcmd{\NAT@citex}
  {\@citea\NAT@hyper@{%
     \NAT@nmfmt{\NAT@nm}%
     \hyper@natlinkbreak{\NAT@spacechar\NAT@@open\if*#1*\else#1\NAT@spacechar\fi}%
       {\@citeb\@extra@b@citeb}%
     \NAT@date}}
  {\@citea\NAT@nmfmt{\NAT@nm}%
   \NAT@spacechar\NAT@@open\if*#1*\else#1\NAT@spacechar\fi\NAT@hyper@{\NAT@date}}
  {}{}

\makeatother

\newcommand{\E}{{\mathbb E}}
\newcommand{\F}{{\mathbb F}}

\renewcommand{\P}{{\mathbb P}}

\newcommand{\C}{{\mathbb C}}
\newcommand{\R}{{\mathbb R}}
\renewcommand{\S}{{\mathbb S}}
\newcommand{\N}{{\mathbb N}}
\newcommand{\Z}{{\mathbb Z}}

\newcommand{\Ecal}{{\mathcal E}}
\newcommand{\Fcal}{{\mathcal F}}

\newcommand{\Scal}{{\mathcal S}}

\newcommand{\Wcal}{{\mathcal W}}

\newcommand{\conv}{\mathrm{conv}}

\newcommand{\id}{{\mathrm{id}}}

\DeclareMathOperator{\interior}{int}
\DeclareMathOperator{\supp}{supp}

\newtheorem{theorem}{Theorem}

\newtheorem{conjecture}[theorem]{Conjecture}

\newtheorem{definition}[theorem]{Definition}
\newtheorem{example}[theorem]{Example}

\newtheorem{lemma}[theorem]{Lemma}

\newtheorem{proposition}[theorem]{Proposition}
\newtheorem{remark}[theorem]{Remark}

\numberwithin{equation}{section}
\numberwithin{theorem}{section}

\begin{document}

\title{Affine processes with compact state space\footnote{Martin Larsson gratefully acknowledges support from SNF Grant 205121\_163425.}}
\author{Paul Kr\"uhner\footnote{University of Liverpool, Institute of Financial \& Actuarial Mathematics, Mathematical Sciences Building, Liverpool, L69 7ZL, Email: peisenbe@liverpool.ac.uk} \quad\quad Martin Larsson\footnote{ETH Zurich, Department of Mathematics, R\"amistrasse 101, CH-8092, Zurich, Switzerland. Email: martin.larsson@math.ethz.ch}}

\maketitle

\begin{abstract}
The behavior of affine processes, which are ubiquitous in a wide range of applications, depends crucially on the choice of state space. We study the case where the state space is compact, and prove in particular that (i) no diffusion is possible; (ii) jumps are possible and enforce a grid-like structure of the state space; (iii) jump components can feed into drift components, but not vice versa. Using our main structural theorem, we classify all bivariate affine processes with compact state space. Unlike the classical case, the characteristic function of an affine process with compact state space may vanish, even in very simple cases.
\\[2ex] 

\noindent{\textbf {Keywords:} Affine processes, Compact state space, Markov chains.}
\\[2ex]
\noindent{\textbf {MSC2010 subject classifications:} 60J25, 60J27, 60J75.}
\end{abstract}

\section{Introduction}

Affine processes are ubiquitous in a wide range of applications, in particular in finance, which has motivated a rich literature developing the mathematical theory of affine processes. We refrain from giving a comprehensive overview here; good starting points include~\cite{duffie/singleton:1999,duffie.al.00,Filipovic:2009fk,piazzesi:2010} and references therein.

Every affine process comes with a state space where it evolves, and the corresponding existence and uniqueness theory depends crucially on the properties of the state space. A complete theory is available for the product space $\R^m\times\R^n_+$ \citep{duffie.al.03} and the convex cone of symmetric positive semidefinite matrices \citep{cuchiero.al.11}, and in the diffusion case for polyhedral and quadratic state spaces \citep{Spreij/Veerman:2012}. Various other state spaces have also been studied, such as symmetric cones \citep{cuchiero/etal:2016}. 
Our focus in the present paper is on compact state spaces, which so far have not received a systematic treatment.

While slightly different definitions of affine processes exist in the literature, they all have the common feature that affine processes are semimartingales (at least before a killing time in the non-conservative case) whose differential characteristics are affine functions of the current state. Since diffusion coefficients must remain nonnegative on the state space and degenerate on the boundary, state spaces which are cones fit well with the affine structure. A similar remark applies to jump intensities. For compact state spaces, as we will see, the affine structure forces the diffusion coefficient to vanish. Jump intensities may however be nonzero, but only if the state space has a grid-like structure along directions where jumps can occur. In particular, the state space may be a finite discrete set, which after an affine transformation only contains points with integer coordinates. Mathematically, this leads to arguments with a combinatorial rather than analytical flavor.

Classically, affine processes have the property that the conditional Fourier transform is an exponential-affine function of the state (indeed, this is sometimes taken as part of the definition of an affine process). In cases where the state space is a finite set, it turns out that the characteristic function may attain the value zero, which precludes the exponential-affine structure. Instead, the characteristic function is a polynomial, where the state appears {\em in the exponents}; this leads to a well-defined expression since, up to an affine transformation, the process has integer coordinates.

The paper is organized as follows. In Section~\ref{S:setup} we introduce the setup and summarize some of our key results. Then, in Section~\ref{S:diff}, we prove that affine processes with compact state space cannot have a diffusive component. In Section~\ref{S:jumps} we discuss the jump (and drift) behavior of affine processes with compact state space. In the final Section~\ref{S:ex} we provide several examples and use the results developed so far to identify all possible affine processes with compact state space in~$\R$ and~$\R^2$. Throughout the paper, $\N=\{0,1,2,\ldots\}$ denotes the natural numbers with zero included. For $d\in\N$, $\R^d$ denotes $d$-dimensional Euclidean space equipped with the usual inner product $\langle\cdot,\cdot\rangle$ and $\mathbb S^d$ denotes the set of symmetric $d\times d$ matrices. In particular, we have $\R^0=\{0\}$. The Dirac measure at a point $x\in\mathbb R^d$ is denoted by~$\delta_x$. Further unexplained notation follows \cite{jacod.shiryaev.03}.

\section{Setup and main results} \label{S:setup}

Fix a measurable space $(\Omega,\Fcal)$ equipped with a right-continuous filtration $\F=(\Fcal_t)_{t\ge0}$. Let~$E$ be a non-empty measurable subset of $\R^d$, $d\ge1$, and let $X=(X_t)_{t\ge0}$ be a c\`adl\`ag adapted process taking values in $E$. We set $X_{0-}=X_0$ by convention. Finally, let $(\P_x)_{x\in E}$ be a family of probability measures on $\Fcal$ such that for each $x\in E$,
\[
\text{$X$ is a $\P_x$-semimartingale with $\P_x(X_0=x)=1$.}
\]
We work with the following notion of an affine process, where $\mathfrak M^d_*$ denotes the vector space of all signed Radon measures on $\R^d\setminus\{0\}$.\footnote{That is, $\mathfrak M^d_*$ is the vector space of all set functions of the form $\mu=\mu_+-\mu_-$, where $\mu_+$ and $\mu_-$ are positive measures on $\R^d\setminus\{0\}$ that assign finite mass to compact subsets.}

\begin{definition} \label{D:affine}
We say that $X$ is {\em affine} if its differential characteristics are affine in the following sense: There exist affine maps
\[
b\colon\R^d\to\R^d, \qquad c\colon\R^d\to\S^d, \qquad F\colon\R^d\to\mathfrak M^d_*
\]
such that for each $x\in E$, $c(x)$ is positive semidefinite, $F(x,d\xi)$ is a positive measure on $\R^d\setminus\{0\}$ satisfying
\begin{equation} \label{eq:Levy}
\int_{\R^d\setminus\{0\}} (\|\xi\|^2\wedge 1) F(x,d\xi)<\infty,
\end{equation}
and the characteristics\footnote{Characteristics of semimartingales are defined in \cite[Definition II.2.6]{jacod.shiryaev.03}} $(B,C,\nu)$ of $X$ under $\P_x$ are given by
\begin{align*}
B_t &= \int_0^t b(X_{s-}) ds, \\
C_t &= \int_0^t c(X_{s-}) ds, \\
\nu(dt,d\xi) &= F(X_{t-},d\xi) dt.
\end{align*}
We refer to $(b,c,F)$ as the {\em triplet} of the affine process $X$.
\end{definition}

We are interested in the following condition on the state space:
\begin{equation*} 
\text{The state space $E$ is compact.} 
\end{equation*}
For notational simplicity, we assume throughout this paper that the affine span of $E$ is all of~$\R^d$. This does not restrict generality; see Remark~2.3 in~\cite{kellerressel/Mayerhofer:2015} for a discussion of this point.

\begin{remark}
The choice of truncation function does not affect Definition~\ref{D:affine}. Changing to a different truncation function yields a different function $b(x)$, but does not affect the affine property. Note also the abuse of terminology: We refer to the process $X$ as affine, although this is rather a property of the family $(\P_x)_{x\in E}$.
\end{remark}

Our convention $X_{0-}=X_0$ implies that the differential characteristics $b(X_{t-})$, $c(X_{t-})$ and $F(X_{t-},d\xi)$ are $\P_x$-almost surely equal to $b(x)$, $c(x)$, and $F(x,d\xi)$, respectively, when $t=0$. By continuity of $b$, $c$, and $F$, the differential characteristics are therefore right-continuous at $t=0$, which will be important when we apply the nonnegativity result Lemma~\ref{L:nonnegativity}.

The following is our main result regarding the structure of affine processes on compact state spaces.

\begin{theorem}\label{T:main}
Let $X$ be affine with triplet $(b,c,F)$, and suppose the state space $E$ is compact. Then there is no diffusion, $c=0$, and there exist an invertible affine map $T\colon\R^d\to\R^d$ as well as $k\in\N$ such that the following conditions hold:
\begin{enumerate}
\item\label{T:main:1} $T(E) \subseteq \mathbb N^k\times \mathbb R^{d-k}$,
\item\label{T:main:2} $Y = (T(X)_1,\dots, T(X)_k)$ is affine and Markov,
\item\label{T:main:3} $Z = (T(X)_{k+1},\dots, T(X)_d)$ can only jump when $Y$ jumps, that is,
\begin{equation} \label{T:main:3_1}
\{t\geq 0\colon \Delta Z(t)\neq 0\} \subseteq \{t\geq 0\colon \Delta Y(t)\neq 0\},
\end{equation}
and its jump characteristic is of the form $\nu^Z(dt,d\zeta)=F^Z(Y_{t-},d\zeta)dt$ for some affine map $F^Z\colon \R^k\to\mathfrak M^{d-k}_*$.
\item\label{T:main:4} the canonical coordinate projections $\pi_j\colon(x_1,\ldots,x_d)\mapsto x_j$ are normalized jump counters of the transformed process $(Y,Z)$ for $j=1,\ldots,k$.\footnote{Normalized jump counters are introduced in Definition~\ref{D:jc} below.}
\end{enumerate}
\end{theorem}

\begin{remark}
Observe that $k$ may be zero in Theorem~\ref{T:main}. Then $Y$ is trivial since it takes values in $\R^0=\{0\}$, and \eqref{T:main:3_1} implies that $Z$ does not jump. Thus $X$ is in fact the (deterministic) solution of a linear ordinary differential equation.
\end{remark}

Theorem~\ref{T:main} combines the conclusions of Theorems~\ref{T:diff=0} and~\ref{T:JM}, whose statements and proofs are given in Sections~\ref{S:diff} and~\ref{S:jumps}, respectively. The proofs yield in some respects more detailed information than the theorem itself; see for instance Lemma~\ref{L:JC_prop} and Proposition~\ref{P:JC repr}. The usefulness of this added detail is illustrated in Section~\ref{S:ex}, where we discuss examples and classify all affine processes on compact state spaces in dimensions $d=1,2$. In particular, Theorem~\ref{t:jumps in 2d}, which treats the case where $E\subseteq\N^2$ is a finite set, has the following corollary:

\begin{theorem} \label{T:d=2}
Let $X=(X_1,X_2)$ be a $2$-dimensional affine process with finite and irreducible state space.\footnote{That $X$ has irreducible state space means that for every $x\in E$, $X$ has positive $\P_x$-probability of eventually reaching any other state $y\in E$.} Assume $X$ has no autonomous components in the sense that no nonzero linear combination $a_1X_1+a_2X_2$ of the components of $X$ is itself an affine process. Then, up to an invertible affine transformation, one has $E=\{x\in\N^2\colon x_1+x_2\le N\}$ for some $N\in\N$, and
\begin{align*}
F(x, \cdot) &= x_1\left( \lambda_1\delta_{(-1,0)} + \lambda_2\delta_{(-1,1)} \right) + x_2\left( \lambda_3\delta_{(0,-1)} + \lambda_4\delta_{(1,-1)} \right) \\
&\quad + (N-x_1-x_2)\left( \lambda_5\delta_{(1,0)} + \lambda_6\delta_{(0,1)} \right)
\end{align*}
for some $\lambda_1,\ldots,\lambda_6\in\R_+$.
\end{theorem}

The characteristic function of $X_t$ in Theorem~\ref{T:d=2} is given by
\[
\E_x[e^{{\bf i} \langle u,X_t\rangle}] = \Phi(u,t) \Psi_1(u,t)^{x_1}\Psi_2(u,t)^{x_2},
\]
where $\Phi(u,t)$ and $\Psi_i(u,t)$, $i=1,2$, are the solutions of the Riccati equations
\begin{align*}
\partial_t\Phi &= N\lambda_5(\Psi_1-1)\Phi + N \lambda_6(\Psi_2-1)\Phi, && \Phi(u,0)=1,\\
\partial_t\Psi_1 &= \lambda_1 - (\lambda_1+\lambda_2-\lambda_5-\lambda_6)\Psi_1 + \lambda_2\Psi_2 -\lambda_6\Psi_1\Psi_2 - \lambda_5\Psi_1^2, && \Psi_1(u,0)=e^{{\bf i}u_1}, \\
\partial_t\Psi_2 &= \lambda_3 - (\lambda_3+\lambda_4-\lambda_5-\lambda_6)\Psi_1 + \lambda_4\Psi_1 -\lambda_5\Psi_1\Psi_2 - \lambda_6\Psi_2^2, && \Psi_2(u,0)=e^{{\bf i}u_2}.
\end{align*}
This follows from the fact that $M_u(t) = \Phi(u,T-t)\Psi_1(u,T-t)^{X_1(t)}\Psi_2(u,T-t)^{X_2(t)}$ defines a martingale for any $T\geq 0$, as can be seen by applying It\^o's formula along with the definition of the characteristics of $X$. Observe that since $E\subset\N^2$, only integer powers appear in the above expressions.

Compare this to the classical case, where the $\Psi_i(u,t)$ are of exponential form and it is their exponents that satisfy (generalized) Riccati equations; see for instance \citet[Theorem 2.7]{duffie.al.03}. In particular, the $\Psi_i(u,t)$ are then necessarily nonzero. \cite{cuchiero.11} also works with $\Psi_i(u,t)$ of exponential form, but allows $\Phi(u,t)$, and hence the characteristic function of $X_t$, to become zero; see equation (1.4) and the subsequent remark in \cite{cuchiero.11}. Our situation with finite state space is different, and it turns out that the $\Psi_i(u,t)$ can in fact reach zero for certain arguments; see Proposition~\ref{p:1d} and the subsequent discussion.

We end this section by noting that any finite state Markov chain can be viewed as an affine process as follows. Let $d$ be the number of states of the Markov chain, let $q_{ij}$ be the intensity of transitioning from state $i$ to state $j$, and let $E=\{e_1,\ldots,e_d\}\subset\R^d$ consist of the canonical unit vectors in $\R^d$. The affine process with state space $E$ and triplet $(0,0,F)$ with $F(x,d\xi)=\sum_{i,j=1}^d e_i^\top x\, q_{ij}\delta_{e_j-e_i}(d\xi)$ then has the same law as the original Markov chain, under the identification of its state space with $E$. This construction leads to an affine process of potentially very large dimension~$d$. If one instead considers a fixed $d$, only some finite-state Markov chains can be viewed as $d$-dimensional affine processes. If $d=1$, the only such Markov chain is the birth--death process; see Proposition~\ref{p:1d}. We thank an anonymous referee for bringing our attention to these observations.

\section{Diffusion} \label{S:diff}

In this section we prove that an affine process with compact state space necessarily has no diffusion.

\begin{theorem} \label{T:diff=0}
Let $X$ be affine with triplet $(b,c,F)$, and suppose the state space $E$ is compact. Then there is no diffusion, i.e., $c=0$.
\end{theorem}

The proof occupies the remainder of this section, and we start with two auxiliary results. The first is an expression of the intuitive notion that there can be no diffusion perpendicularly to the boundary of $E$.

\begin{lemma} \label{L:c(x) 1}
Let the assumptions of Theorem~\ref{T:diff=0} be in force. Let $u\in\R^d$, $\overline x\in E$, and $\langle u,\overline x\rangle=\max_{x\in E}\langle u,x\rangle$. Then $c(\overline x)u=0$. Note that the set of possible maximizers $\overline x$ depends on the choice of $u$.
\end{lemma}

\begin{proof}
Define $f(x)=\langle u,\overline x-x\rangle$. Then under $\P_{\overline x}$ the process $Z=f(X)$ is a nonnegative semimartingale with $Z_0=0$. Its second differential characteristic is \cite[Proposition B.1]{kallsen.kruehner.es.15}
\[
\widetilde c_t = \langle \nabla f(X_{t-}),c(X_{t-})\nabla f(X_{t-})\rangle = \langle u, c(X_{t-})u\rangle, \qquad \widetilde c_0=\langle u,c(\overline x)u\rangle.
\]
Lemma~\ref{L:nonnegativity} implies $\widetilde c_0=0$, and hence $c(\overline x)u=0$.
\end{proof}

Let $\conv(E)$ denote the convex hull of $E$, which is again compact. Since the affine span of $E$, and hence of $\conv(E)$, is all of $\R^d$, we have
\[
\interior \conv(E) \ne\emptyset.
\]
The next result is an application of the \cite{caratheodory.07} theorem, allowing us to replace $E$ by $\conv(E)$ in Lemma~\ref{L:c(x) 1}. By definition, this means that $c$ is parallel to $\conv(E)$; see~\eqref{eq:parallel def}.

\begin{lemma} \label{L:c is parallel to K}
Let the assumptions of Theorem~\ref{T:diff=0} be in force. Then $c$ is parallel to $\conv(E)$.
\end{lemma}

\begin{proof}
Pick any nonzero $\overline x\in \conv(E)$ and any $u\in N_{\conv(E)}(\overline x)$. This means that $\langle u,\overline x\rangle=\max_{x\in \conv(E)}\langle u,x\rangle$. We must show that $c(\overline x)u=0$. By Carath\'eodory's theorem \cite[Theorem~17.1]{rockafellar.70}, $\overline x$ can be expressed as a convex combination of $d+1$ points in $E$. Thus there exist $k\le d+1$, $x_1,\ldots,x_k \in E$, and $\lambda_1,\ldots,\lambda_k \in (0,1)$ such that $\overline x=\sum_{i=1}^k\lambda_ix_i$. Since
\[
\max_i \langle u, x_i\rangle \ge \sum_{i=1}^k \lambda_i\langle u, x_i\rangle = \langle u,\overline x\rangle = \max_{x\in K}\langle u,x\rangle \ge \max_{x\in E}\langle u,x\rangle \ge \max_i \langle u,x_i\rangle,
\]
and since the $\lambda_i$ are strictly positive, one has $\langle u,x_i\rangle = \max_{x\in E}\langle u, x\rangle$ for all~$i$. Thus Lemma~\ref{L:c(x) 1} yields $c(x_i)u=0$ for all~$i$, whence, on taking convex combinations, $c(\overline x)u=0$.
\end{proof}

The following result is the key ingredient in the proof of Theorem~\ref{T:diff=0}. Its proof relies on notions and results from convex analysis that are developed in Section~\ref{APP:convex}.

\begin{proposition} \label{P:c=0}
Let $K$ be a compact convex subset of $\R^d$ with $\interior K\ne\emptyset$. Let $c:\R^d\to\S^d$ be an affine map parallel to~$K$ with $c(x)$ positive semidefinite for all $x\in K$. Then $c=0$.
\end{proposition}

For the proof of this proposition it is convenient to adopt the following coordinate-free notation: For vector spaces $V$ and $W$, we let $S(V)$ denote the space of symmetric linear maps $V\to V$, and $L(V,W)$ the space of linear maps $V\to W$. Thus $S(V)\simeq\S^n$ and $L(V,W)\simeq\R^{m\times n}$, where $n=\dim V$ and $m=\dim W$.

\begin{proof}
After applying a translation, we may suppose that $0\in\interior K$. Then $c(x)$ is of the form $c(x)=c_0+\ell(x)$ for some $c_0\in\S^d_+$ and some linear map $\ell:\R^d\to\S^d$. Consider the orthogonal direct sum decomposition
\[
\R^d = V\oplus W,
\]
where $W=\ker c_0$ and $V=W^\perp$ is the range of $c_0$. With respect to this decomposition, $c(x)$ takes the form
\[
c(x) = c_0 + L(x) = \begin{pmatrix}
a_0	& 0 \\ 
0	& 0
\end{pmatrix}
+
\begin{pmatrix}
\ell_{11}(x) & \ell_{12}(x) \\
\ell_{21}(x) & \ell_{22}(x)
\end{pmatrix},
\]
where $a_0\in S(V)$ is positive definite and the maps $\ell_{11}:\R^d\to S(V)$, $\ell_{12}:\R^d\to L(W,V)$, and $\ell_{22}:\R^d\to S(W)$ are all linear.

Consider any $w\in W$. The map $x\mapsto\langle w,c(x)w\rangle=\langle w,\ell_{22}(x)w\rangle$ is linear and nonnegative in some open ball $B\subset K$ around the origin, and is thus identically zero. Since $c(x)$, and hence $\ell_{22}(x)$, is positive semidefinite for all $x\in B$, we deduce $\ell_{22}(x)=0$ for all $x\in\R^d$. Again by positive semidefiniteness of $c(x)$ for $x\in B$, this implies $\ell_{12}(x)=0$ for all $x\in B$, and hence for all $x\in\R^d$.

We now show that $\dim V=0$; this will imply that $c=\ell_{22}=0$ and thus complete the proof. To this end, define
\[
\widehat K = V\cap K,
\]
a compact convex subset of $V$ with $0\in\interior \widehat K$, where the interior is to be understood relative to $V$. Consider also the linear map
\[
\widehat c:V \to S(V), \qquad y \mapsto a_0 + \ell_{11}(y).
\]
We claim that $\widehat c$ is parallel to $\widehat K$. To see this, pick any $y\in\widehat K$, $v\in N_{\widehat K}(y)$. Lemma~\ref{L:K hat} then yields $v+w\in N_K(y)$ for some $w\in W$. Consequently, using that $c(x)w=0$ and that $c$ is parallel to~$K$ by assumption,
\[
\widehat c(y) v = c(y)(v+w) = 0.
\]
Thus $\widehat c$ is parallel to $\widehat K$ as claimed. Furthermore, $\widehat c$ is affine, positive semidefinite on $\widehat K$, and $\widehat c(0)=a_0$ is invertible. Since $\widehat K$ is compact, Lemma~\ref{L:K unbounded} implies $\dim V=0$, as required.
\end{proof}

The proof of Theorem~\ref{T:diff=0} is now straightforward.

\begin{proof}[Proof of Theorem~\ref{T:diff=0}]
Since $c$ is affine, the positive semidefiniteness of $c(x)$ for $x\in E$ in fact holds for all $x\in \conv(E)$. Moreover, $c$ is parallel to $\conv(E)$ by Lemma~\ref{L:c is parallel to K}. The result now follows from Proposition~\ref{P:c=0} with $K=\conv(E)$.
\end{proof}

\section{Jumps and drift} \label{S:jumps}

Assume that $X$ is affine with triplet $(b,c,F)$ and compact state space $E\subset\R^d$ whose affine span is all of $\R^d$. Since $c=0$ by Theorem~\ref{T:diff=0} we are left with $(b,0,F)$, and our goal is to describe its structure. This will be done through a sequence of intermediate results culminating with Theorem~\ref{T:JM} below.

It is convenient to introduce the set
\begin{equation}\label{eq:S}
S=\bigcup_{x\in E} \supp (F(x,\cdot)) \setminus \{0\},
\end{equation}
which can be thought of as the collection of all possible jump sizes of $X$. If $X$ does not jump at all, then $S=\emptyset$.

A key property of $F$, beyond its affine structure, is that $F(x,A)=0$ whenever $x\in E$ and $A\cap(E-x)=\emptyset$, which holds because $F(X_{t-},d\xi)dt$ is the jump characteristic of $X$. A more useful way to phrase this property is
\[
\text{$x\in E$, $A\subseteq\R^d$ open, $F(x,A)>0$} \quad\Longrightarrow\quad \text{$x+\xi\in E$ for some $\xi\in A$,}
\]
which we will use repeatedly without explicit mentioning. Together with compactness of $E$ this leads to the existence of {\em jump counters}, which we now define.

\begin{definition} \label{D:jc}
A {\em jump counter} corresponding to $u\in S$ is an affine function $\psi_u:\R^d\to\R$, not identically zero, such that $\psi_u\ge 0$ on $E$ and
\[
\text{$x\in E$ and $\psi_u(x)>0$} \quad\Longrightarrow\quad x+u\in E.
\]
We call $\psi_u$ {\em normalized} if $\psi_u(u)-\psi_u(0)=-1$.
\end{definition}

Jump counters are useful because they have rather strong implications for the structure of $E$. To see why, let $\psi_u$ be a jump counter corresponding to some vector $u\in S$. Each point $x\in E$ satisfies either $\psi_u(x)=0$ or $\psi_u(x)>0$. In the latter case, $x+u$ again lies in $E$. Then by the same token, either $\psi_u(x+u)=0$, or $x+2u\in E$. Iterating this argument and observing that the procedure must terminate in finitely many steps since $E$ is compact, we obtain, for every $x\in E$,
\begin{equation} \label{eq:psi_u step}
\{x,x+u,\ldots,x+nu\}\subseteq E \quad\text{and}\quad \psi_u(x+nu)=0 \quad \text{for some $n\in\N$.}
\end{equation}
Since $\psi_u$ is not identically zero and $E$ affinely spans $\R^d$, we have $\psi_u(x)>0$ for some $x\in E$. Thus \eqref{eq:psi_u step} implies that $\psi_u$ is strictly decreasing in direction $u$, so that $\psi_u(0)-\psi_u(u)>0$. Therefore, we can always replace $\psi_u$ by $\psi_u/(\psi_u(0)-\psi_u(u))$ to obtain a normalized jump counter.

\begin{lemma} \label{L:JC_prop}
Any normalized jump counter $\psi_u$ corresponding to some $u\in S$ satisfies
\begin{enumerate}
\item\label{L:JC_prop:1} $\psi_u(E)\subseteq\N$,
\item\label{L:JC_prop:2} $\psi_u(x)=1$ for some $x\in E$,
\item\label{L:JC_prop:3} $\ker \psi_u=\mathrm{aff}\{x_1,\dots,x_d\}$ for some $x_1,\dots,x_d\in E$.
\end{enumerate}
Moreover, if $\phi_u$ is any other normalized jump counter corresponding to $u$, then $\phi_u=\psi_u$.
\end{lemma}

\begin{proof}
For any $x\in E$ and with $n\in\N$ as in \eqref{eq:psi_u step} we have
\[
\psi_u(x) = \psi_u(x)-\psi_u(x+nu) = \psi_u(0) - \psi_u(nu) = n (\psi_u(0) - \psi_u(u)) = n,
\]
where the affine property of $\psi_u$ is used in the second and third equalities. This yields \ref{L:JC_prop:1}. Next, with $x\in E$ such that $\psi_u(x)>0$, a similar calculation gives $\psi_u(x+(n-1)u)=1$ which proves \ref{L:JC_prop:2}. We now argue \ref{L:JC_prop:3}. Since $\psi_u$ is strictly decreasing along $u$, the intersection $(x+\R u)\cap\ker\psi_u$ contains exactly one element, where we write $x+\mathbb Ru = \{ x+tu: t\in\mathbb R\}$. The projection along $u$ onto $\ker\psi_u$ is therefore given by
\[
\pi:\R^d\to\ker\psi_u, \quad x\mapsto (x+\R u)\cap\ker\psi_u,
\]
where we identify the intersection on the right-hand side with the single element it contains.
In particular, for any $x\in E$ and with $n\in\N$ as in \eqref{eq:psi_u step} we have $\pi(x)=x+nu \in E$ for some $n\in\N$, so that $\pi(E)\subseteq E$. Since the affine span of $E$ is all of $\R^d$, the image $\pi(E)$ affinely spans $\ker\psi_u$. Thus $E$ contains $d$ points $x_1,\ldots,x_d$ whose affine span is $\ker\psi_u$, proving \ref{L:JC_prop:3}.

It remains to prove the uniqueness statement. For any $x\in E\cap\ker\psi_u$ we have $\psi_u(x+u)=\psi_u(x+u)-\psi_u(x)=-1$, and hence $x+u\notin E$. Therefore $\phi_u(x)=0$ by the definition of a jump counter. Letting $x_1,\ldots,x_d\in E$ affinely span $\ker\psi_u$, we obtain
\[
\ker \psi_u = \mathrm{aff}\{x_1,\dots,x_d\} \subseteq \ker \phi_u.
\]
Thus there exists a constant $\lambda$ such that $\phi_u = \lambda \psi_u$, and since both $\psi_u$ and $\phi_u$ are normalized we get $1 = \phi_u(0)-\phi_u(u) = \lambda(\psi_u(0)-\psi_u(u)) = \lambda$.
\end{proof}

In Proposition~\ref{P:JC repr} below we obtain the existence of normalized jump counters and show that they have additional properties. The proof uses the following two lemmas.

\begin{lemma}\label{L:J1}
Let $\Wcal^{d-1}$ denote the set of all affine subspaces $W\subseteq\R^d$ with $\dim W\le d-1$. Then
\[
\inf_{W\in\Wcal^{d-1}} \sup_{x\in E} d(x,W) > 0,
\]
where $d(x,W)=\inf\{\|x-y\|\colon y\in W\}$ is the distance from $x$ to $W$.
\end{lemma}

\begin{proof}
For any $W\in\Wcal^{d-1}$ and $x\in\R^d$ there exist $y\in W$ and $u\in (W-y)^\perp$ with $\|u\|=1$ such that $d(x,W)=|\langle u,x-y\rangle|$. Thus
\[
\inf_{W\in\Wcal^{d-1}} \sup_{x\in E} d(x,W) = \inf\left\{ \sup_{x\in E} |\langle u,x-y\rangle| \colon (u,y) \in \Scal^{d-1}\times \R^d \right\},
\]
where $\Scal^{d-1}$ is the unit sphere in $\R^d$. Since $E$ is compact, the map $(u,y)\mapsto\sup_{x\in E} |\langle u,x-y\rangle|$ is continuous, and it suffices to let $y$ range over a compact subset $K\subseteq\R^d$. Thus the infimum is attained for some $(\overline u,\overline y)\in\Scal^{d-1}\times K$, so that
\[
\inf_{W\in\Wcal^{d-1}} \sup_{x\in E} d(x,W) = \sup_{x\in E} |\langle \overline u,x-\overline y\rangle|.
\]
The right-hand side is strictly positive since the affine span of $E$ is all of $\R^d$.
\end{proof}

\begin{lemma}\label{L:J2}
For every $u\in S$ there exist a constant $\varepsilon>0$ and a probability measure $\lambda$ on $\R^d$ such that
\begin{equation}\label{L:J2:1}
F(\cdot, A\cap B_{\varepsilon}(u)) = \lambda(A) F(\cdot, B_{\varepsilon}(u)) \quad\text{for all measurable $A\subseteq\R^d$},
\end{equation}
where $B_{\varepsilon}(u)$ denotes the open ball with radius $\varepsilon$ centered at $u$. Furthermore, if $x\in E$ and $F(x, B_{\varepsilon}(u))>0$, then $x+u\in E$. In particular, $\phi_u=F(\cdot, B_{\varepsilon}(u))$ is a jump counter.
\end{lemma}

\begin{proof}
Define the diameter of $E$ by $\mathrm{diam}(E)=\sup\{\|x-y\|\colon x,y\in E\}>0$.  Let $N\in \mathbb N$ such that $N>(\mathrm{diam}(E)+1)/\|u\|$ and pick $\delta>0$ such that
\[
\delta < \inf_{W\in\Wcal^{d-1}} \sup_{x\in E} d(x,W),
\]
where $\Wcal^{d-1}$ denotes the set of all affine subspaces $W\subseteq\R^d$ with $\dim W\le d-1$; this is possible by Lemma~\ref{L:J1}. Set
\[
\varepsilon = \frac{\delta\wedge 1}{N} \wedge \frac{\|u\|}{2} > 0,
\]
and define the affine functions $\phi_u = F(\cdot,B_\varepsilon(u))$ and $\phi_A = F(\cdot, A\cap B_\varepsilon(u))$ for any measurable subset $A\subseteq\R^d$. These functions are finite-valued since $B_\varepsilon(u)$ is bounded away from the origin. Consider the affine subspace
\[
V = \ker \phi_u \cap \ker \phi_A.
\]
 We claim that
\begin{equation} \label{eq:L:J1:1}
\text{$V\ne\emptyset$ and $V\ne\R^d$.}
\end{equation}
To see this, note that $u\in\supp F(x_0,\cdot)$ for some $x_0\in E$, and hence $\phi_u(x_0) = F(x_0,B_{\varepsilon}(u))>0$. Thus there exists $x_1\in E$ with $x_1-x_0\in B_\varepsilon(u)$. Then, recursively, if $x_j\in E$ satisfies $\phi_u(x_j)>0$, we find $x_{j+1}\in E$ with $x_{j+1}-x_j\in B_\varepsilon(u)$. Note that $\|x_j-x_0\|\ge j(\|u\|-\varepsilon)\ge j\|u\|/2$. Therefore by compactness of $E$ there is a maximal $j$ such that $x_j\in E$ and $\phi_u(x_j)>0$, and for this $j$ we have $\phi_u(x_{j+1})\le0$. But $F(x,\cdot)$ is a positive measure for all $x\in E$, so we deduce $0\le\phi_A(x_{j+1})\le\phi_u(x_{j+1})\le0$. This completes the proof of~\eqref{eq:L:J1:1}.

Next, we claim that
\begin{equation} \label{eq:L:J1:2}
\ker\phi_u \subseteq \ker\phi_A.
\end{equation}
If $\phi_A\equiv0$, this certainly holds. Otherwise \eqref{eq:L:J1:1} implies $\dim V\le d-1$ and, in case of equality, $\ker\phi_u=\ker\phi_A$. It remains to exclude the possibility that $\dim V \le d-2$, so we assume for contradiction that this holds. Then the affine subspace
\[
W = V + \R u
\]
satisfies $\dim W\le d-1$. By definition of $\delta$, we can then find $x_0\in E$ such that $d(x_0,W)\ge\delta$. In particular $x_0\notin W$, whence either $\phi_u(x_0)>0$ or $\phi_A(x_0)>0$ (or both). Thus there exists $x_1\in E$ with $x_1-x_0\in B_\varepsilon(u)$, and hence
\[
d(x_1,W) = d(x_1-u,W) \ge d(x_0,W) - d(x_0,x_1-u) \ge \delta - \varepsilon \ge \delta\left(1 - \frac1N\right).
\]
Again we proceed recursively: If $x_j\in E$ satisfies $ d(x_j,W)\ge \delta (1-j/N)>0$, we find $x_{j+1}\in E$ with $x_{j+1}-x_j\in B_\varepsilon(u)$ and $ d(x_{j+1},W)\ge \delta(1-(j+1)/N)$. Consequently,
\[
\|x_N - x_0\| = \| N u + \sum_{j=0}^{N-1} (x_{j+1} - x_j - u) \| \ge N\|u\| - N \varepsilon > \mathrm{diam}(E),
\]
a contradiction. Thus~\eqref{eq:L:J1:2} is proved.

Next, \eqref{eq:L:J1:2} implies that $\phi_A = \lambda(A) \phi_u$ for some constant $\lambda(A)$ that depends on $A$, which proves \eqref{L:J2:1}. The fact that $\lambda$ is a probability measure follows by inspecting \eqref{L:J2:1} at a point $x\in E$ for which $F(x,B_\varepsilon(u))=\phi_u(x)>0$.

Finally, to prove the last statement, consider $x\in E$ such that $F(x, B_{\varepsilon}(u))>0$. Define $A_n=B_{n^{-1}\wedge\varepsilon}(u)$, so that
\[
F(\cdot, A_n) = \lambda(A_n)F(\cdot,B_\varepsilon(u)) \quad\text{for all large $n$.}
\]
Since $u\in S$, the left-hand side is not identically zero, and so $\lambda(A_n)>0$ for all large $n$. Evaluating at $x$ then yields $F(x,A_n)>0$, and thus there exist $u_n\in A_n$ such that $x+u_n\in E$. Since $u_n\to u$ and $E$ is closed, it follows that $x+u\in E$ as claimed.
\end{proof}

The following proposition shows that each jump size $u\in S$ admits a normalized jump counter with additional properties. Furthermore, the proposition gives information on how jump counters $\psi_u$ and $\psi_v$ corresponding to different $u,v\in S$ interact. Informally, $\psi_u(x)$ counts the number of jumps of size $u$ the process $X$ can perform, starting from $x\in E$, until it reaches $\ker\psi_u$. Now, if $\psi_u(x+v)<\psi_u(x)$, then a jump of size $v$ will bring $X$ closer to $\ker\psi_u$, and the proposition shows that then, in fact, $\psi_u=\psi_v$. If instead $\psi_u(x+v)=\psi_u(x)$, then the jump $v$ is parallel to $\ker\psi_u$. This implies that $\ker\psi_u$ and $\ker\psi_v$ have a nonempty intersection, where jumps of size $u$ or $v$ do not occur. Finally, if $\psi_u(x+v)>\psi_u(x)$ then a jump of size $v$ moves $X$ farther away from $\ker\psi_u$, and in this case it turns out that $v=-u$.

\begin{proposition} \label{P:JC repr}
Every $u\in S$ admits a unique normalized jump counter $\psi_u$. The jump counter satisfies
\begin{equation} \label{P:JC repr:1}
F(x, A\cap B_{\varepsilon_u}(u)) = \lambda_u(A)\psi_u(x) \quad\text{for all $x\in E$ and all measurable $A\subseteq\R^d$},
\end{equation}
where $\lambda_u$ is a finite measure on $\R^d$ and $\varepsilon_u>0$ is a constant. Moreover, for any $u,v\in S$ and setting $\alpha=\psi_u(v)-\psi_u(0)$ and $\beta=\psi_v(u)-\psi_v(0)$, one of the following conditions holds:
\begin{enumerate}
\item\label{P:JC repr:alt1} $\alpha=\beta=-1$ and $\psi_u=\psi_v$,
\item\label{P:JC repr:alt2} $\alpha=\beta=1$ and $u=-v$,
\item\label{P:JC repr:alt3} one of $\alpha$ and $\beta$ equals zero, and $\alpha,\beta\in\N$.
\end{enumerate}
\end{proposition}

\begin{proof}
Existence of a jump counter $\psi_u$ satisfying~\eqref{P:JC repr:1} follows from Lemma~\ref{L:J2}. Since \eqref{P:JC repr:1} is preserved after positive scaling of $\psi_u$, we may assume that $\psi_u$ is normalized as discussed before Lemma~\ref{L:JC_prop}. This also yields uniqueness.

Consider now $u,v\in S$ with jump counters $\psi_u$, $\psi_v$ and $\alpha$, $\beta$ as stated. There is some $x\in E$ with $\psi_v(x)>0$ and hence $x+v\in E$. Thus $\alpha=\psi_u(v)-\psi_u(0)=\psi_u(x+v)-\psi_u(x)\in\Z$ due to Lemma~\ref{L:JC_prop}\ref{L:JC_prop:1}, and similarly $\beta\in\Z$ as well. We proceed by examining the possible values of $\alpha$ and $\beta$.

{\em Case 1:} $\alpha<0$ or $\beta<0$. Suppose $\alpha<0$. Lemma~\ref{L:JC_prop}\ref{L:JC_prop:3} yields points $x_1,\ldots,x_d\in E$ that affinely span $\ker\psi_u$. Moreover,
\[
\psi_u(x_j+v) = \psi_u(x_j+v)-\psi_u(x_j) = \psi_u(v)-\psi_u(0) = \alpha < 0,
\]
and hence $x_j+v\notin E$. We conclude that $\psi_v(x_j)=0$ for $j=1,\ldots,d$, whence $x_1,\ldots,x_d$ affinely span $\ker\psi_v$. Thus the kernels of $\psi_u$ and $\psi_v$ coincide, so that $\psi_v=\lambda\psi_u$ for some constant $\lambda$. Observing that
\[
-1 = \psi_v(v)-\psi_v(0) = \lambda(\psi_u(v)-\psi_u(0)) = \lambda\alpha
\]
we have $\lambda=-1/\alpha$, and therefore
\[
-1 = \psi_u(u)-\psi_u(0) = -\alpha(\psi_v(u)-\psi_v(0)) = -\alpha\beta.
\]
Consequently $\alpha=\beta=-1$ and $\psi_u=\psi_v$, so that \ref{P:JC repr:alt1} holds. The same argument yields the conclusion when $\beta<0$.

{\em Case 2:} $\alpha>0$ and $\beta>0$. From Lemma~\ref{L:JC_prop}\ref{L:JC_prop:2} there is $x\in E$ such that $\psi_v(x) = 1$. Define
\[
J=\{j\in\mathbb N\colon x+j(u+v)\in E \text{ and } \psi_v(x+j(u+v))>0\}.
\]
We claim that $J=\N$. To see this, first note that $0\in J$ since $\psi_v(x) > 0$. Moreover, for any $j\in J$ we have $x+ju + (j+1)v\in E$ and thus $\psi_u(x+ju + (j+1)v) = \psi_u(x+j(u+v)) + \alpha > 0$. Consequently, $x+(j+1)(u+v)\in E$ and $\psi_v(x+(j+1)(u+v)) = \psi_v(x + ju + (j+1)v) + \beta >0$. That is, $j+1\in J$ and hence $J=\N$ by induction. Compactness of $E$ then forces $u+v=0$, thus $\alpha=\psi_u(-u)-\psi_u(0)=\psi_u(0)-\psi_u(u)=1$, and similarly $\beta=1$. Thus \ref{P:JC repr:alt2} holds.

Case 3: $\alpha=0$ and $\beta\ge0$, or $\alpha\ge0$ and $\beta=0$. This directly gives \ref{P:JC repr:alt3} since $\alpha$ and $\beta$ are integers.
\end{proof}

Recall that $\mathfrak M^d_*$ denotes the vector space of all signed Radon measures on $\R^d\setminus\{0\}$.

\begin{lemma} \label{L:F1}
Assuming that $S\ne\emptyset$, there exist vectors $u_1,\ldots,u_k\in S$ and an affine map $F_1\colon\R^k\to\mathfrak M^d_*$ such that
\begin{enumerate}
\item\label{L:F1:1} the affine map $\Psi\colon\R^d\to\R^k$ given by $\Psi(x)=(\psi_{u_1}(x),\ldots,\psi_{u_k}(x))$ is surjective,
\item\label{L:F1:2} $F_1(\Psi(x),\cdot) = F(x,\cdot)$ for all $x\in \R^d$,
\item\label{L:F1:3} every normalized jump counter $\psi_u$, $u\in S$, is of the form
\[
\psi_u = c_0+c_1\psi_{u_1}+\cdots+c_k\psi_{u_k}
\]
for some $c_0,\ldots,c_k\in\R$.
\end{enumerate}
Here the $\psi_{u_j}$ are the normalized jump counters corresponding to the $u_j$.
\end{lemma}

\begin{proof}
We first claim that
\begin{equation} \label{eq:m30}
{\rm span}\{\psi_u \colon u\in S\} = {\rm span}\{ F(\cdot,A)\colon A\subseteq\R^d \text{ measurable}\}.
\end{equation}
The inclusion $\supseteq$ follows immediately from \eqref{P:JC repr:1}. We now prove the reverse inclusion $\subseteq$. For each $u\in S$, let $\lambda_u$ and $\varepsilon_u$ be as in Proposition~\ref{P:JC repr}. Then $\{ B_{\varepsilon_u}(u):u\in S\}$ is an open covering of $S$. Since any open covering of any subset of $\R^d$ admits a countable subcovering, we may choose countably many vectors $u_j\in S$, $j\in\mathbb N$, such that $S\subseteq \bigcup_{j=1}^\infty B_{\varepsilon_{u_j}}(u_j)$. Defining $C_1=B_{\varepsilon_{u_1}}(u_1)$ and then recursively $C_j=B_{\varepsilon_{u_j}}(u_j)\setminus \bigcup_{i<j} C_i$ we obtain a pairwise disjoint countable measurable covering of $S$ with $C_j\subseteq B_{\varepsilon_{u_j}}(u_j)$ for all $j\in \mathbb N$. For any measurable $A\subseteq\R^d$ we then have
\[
F(x,A) = \sum_{j=1}^\infty F(x,A\cap C_j) = \sum_{j=1}^\infty \lambda_{u_j}(A\cap C_j) \psi_{u_j}(x), \quad x\in\mathbb R^d.
\]
Since ${\rm span}\{\psi_u \colon u\in S\}$ is a subspace of the $(d+1)$-dimensional space of all affine functions from $\R^d$ to $\R$, it is closed under pointwise convergence and therefore contains $F(\cdot,A)$. This proves \eqref{eq:m30}.

Write $V={\rm span}\{\psi_u \colon u\in S\}$ for brevity. If $V$ contains the constant function $1$, we can find $u_1,\ldots,u_k$ such that $\{1,\psi_{u_1},\ldots,\psi_{u_k}\}$ is a basis for $V$. If $V$ does not contain $1$, we can find $u_1,\ldots,u_k$ such that $\{\psi_{u_1},\ldots,\psi_{u_k}\}$ is a basis for $V$. In either case, the $\psi_{u_j}$ along with $1$ are linearly independent, and we have \ref{L:F1:3}.

We also obtain \ref{L:F1:1}. Indeed, if $c_0 + c_1\psi_{u_1} + \cdots + c_k\psi_{u_k} = 0$ for some constants $c_0,\ldots,c_k$, then $c_0=c_1=\cdots=c_k=0$ by linear independence. Thus the affine space $\Psi(\R^d)$ is not contained in any proper affine subspace of $\R^k$, and so must be all of $\R^k$, which proves~\ref{L:F1:1}.

Since $\Psi$ is surjective, there exists an affine map $\Phi:\R^k\to\R^d$ such that $\Psi\circ\Phi=\id$. We define $F_1:\R^k\to\mathfrak M^d_*$ by
\[
F_1(y,A) = F(\Phi(y),A).
\]
This is affine in $y$, being the composition of two affine maps. We now argue \ref{L:F1:2}. Pick $x\in\R^d$ and set $y=\Psi(x)$ and $x'=\Phi(y)$. It follows from \eqref{eq:m30} and the choice of $u_1,\ldots,u_k$ that for any fixed measurable subset $A\subseteq\R^d$, there exist $c_0\in\R$ and $c\in\R^k$ such that $F(\cdot,A)=c_0+\langle c, \Psi(\cdot)\rangle$. Since also $\Psi(x')=\Psi\circ\Phi(y)=y=\Psi(x)$, we get
\[
F_1(y,A) = F(x',A) = c_0+\langle c, \Psi(x')\rangle = c_0+\langle c, \Psi(x)\rangle = F(x,A).
\]
This proves \ref{L:F1:2}.
\end{proof}

The affine map $\Psi\colon\R^d\to\R^k$ in Lemma~\ref{L:F1}, being surjective, can be extended to an invertible affine map $T\colon \R^d\to\R^d$ whose first $k$ component functions $T_1,\ldots,T_k$ are precisely $\Psi_1,\ldots,\Psi_k$. In particular, $T_j=\psi_{u_j}$ is a normalized jump counter for each $j=1,\ldots,k$. Note that $k=0$ is possible, and occurs precisely when $X$ does not jump at all and thus $S=\emptyset$. In this case $T$ is simply an arbitrary invertible affine map, for example the identity map.

\begin{theorem}\label{T:JM}
The invertible affine map $T$ satisfies the following properties:
\begin{enumerate}
\item\label{T:JM:1} $T(E) \subseteq \mathbb N^k\times \mathbb R^{d-k}$,
\item\label{T:JM:2} $Y = (T(X)_1,\dots, T(X)_k)$ is affine and Markov,
\item\label{T:JM:3} $Z = (T(X)_{k+1},\dots, T(X)_d)$ can only jump when $Y$ jumps, that is,
\begin{equation} \label{T:JM:3_1}
\{t\geq 0\colon \Delta Z(t)\neq 0\} \subseteq \{t\geq 0\colon \Delta Y(t)\neq 0\} \text{ a.s.},
\end{equation}
and its jump characteristic is of the form $\nu^Z(dt,d\zeta)=F^Z(Y_{t-},d\zeta)dt$ for some affine map $F^Z\colon \R^k\to\mathfrak M^{d-k}_*$,
\item\label{T:JM:4} the canonical coordinate projections $\pi_j\colon(x_1,\ldots,x_d)\mapsto x_j$ are normalized jump counters of the transformed process $(Y,Z)$ for $j=1,\ldots,k$.
\end{enumerate}
\end{theorem}

\begin{proof}
Part \ref{T:JM:1} follows directly from the construction of $T$ and Lemma~\ref{L:JC_prop}\ref{L:JC_prop:1}. To prove \ref{T:JM:2}, let $F_1$ be as in Lemma~\ref{L:F1} and define
\[
F^Y(y,B) = F_1(y, \Psi^{-1}(\Psi(0) + B))
\]
for each $y\in\R^k$ and measurable $B\subseteq \R^k\setminus\{0\}$. For any $y=\Psi(x)$ with $x\in E$, the definition of $F^Y$ along with Lemma~\ref{L:F1}\ref{L:F1:2} yield
\begin{align*}
\int_{\R^k\setminus\{0\}} (\|\eta\|^2\wedge 1) F^Y(y,d\eta)
&= \int_{\{\xi\in\R^d\colon \Psi(\xi)\ne\Psi(0)\}} (\|\Psi(\xi)-\Psi(0)\|^2\wedge 1)F_1(y,d\xi) \\
&\le \kappa^2 \int_{\R^d\setminus\{0\}} (\|\xi\|^2\wedge 1) F(x,d\xi) \\
&<\infty,
\end{align*}
where $\kappa$ is the maximum of one and the operator norm of the linear map $\Psi-\Psi(0)\colon\R^d\to\R^k$. Thus, for all $y\in\Psi(E)$,  $F^Y(y,\cdot)$ is a positive measure on $\R^k\setminus\{0\}$ satisfying \eqref{eq:Levy}, and in particular lies in $\mathfrak M^k_*$. Since $E$ affinely spans $\R^d$ and $\Psi$ is surjective, it follows that $\Psi(E)$ affinely spans $\R^k$. From the affine dependence on $y$ we then infer $F(y,\cdot)\in\mathfrak M^k_*$ for all $y\in\R^k$. Thus $F^Y$ qualifies as the jump kernel of an affine process with state space $\Psi(E)$.

Next, the process $Y=\Psi(X)$ is an $\N^k$-valued semimartingale with characteristics $(B^Y,C^Y,\nu^Y)$, say. Since it takes values in a discrete set, we may take $B^Y=0$ and $C^Y=0$ (the latter also follows from the fact that $X$ has no diffusion part). In view of \cite[Proposition B.1]{kallsen.kruehner.es.15} along with Lemma~\ref{L:F1}\ref{L:F1:2}, the jump characteristic $\nu^Y$ is given by
\begin{align*}
\nu^Y([0,t]\times B) &= \int_0^t \int_{\R^d\setminus\{0\}} \bm 1_{\{\Psi(X_{s-}+\xi)-\Psi(X_{s-})\in B\}} F(X_{s-},d\xi)ds \\
&= \int_0^t \int_{\R^d\setminus\{0\}} \bm 1_{\{\Psi(\xi)-\Psi(0)\in B\}} F_1(\Psi(X_{s-}),d\xi)ds \\
&= \int_0^t F^Y(Y_{s-},B)ds
\end{align*}
for any $t\ge0$ and measurable $B\subseteq\R^k\setminus\{0\}$. We conclude that $Y$ is an affine process with state space $\Psi(E)$ and triplet $(0,0,F^Y)$. Since the state space is finite the Markov property follows easily, for instance by an argument based on \cite[Theorem~4.4.1]{ethier.kurtz.86}. This completes the proof of~\ref{T:JM:2}.

We now prove \ref{T:JM:3}. We claim that there exists a nullset $N\subseteq\Omega$, such that for every $\omega\notin N$ one has $\Delta X_t(\omega) \in S\cup\{0\}$ for all $t\ge 0$. Indeed, since the jump characteristic of $X$ is $F(X_{t-},d\xi)dt$, \citet[Theorem~II.1.8]{jacod.shiryaev.03} yields
\[
\E\left[ \sum_{t>0}\bm1_{\R^d\setminus(S\cup\{0\})}(\Delta X_t)\right] = \E\left[\int_0^\infty F(X_{t-},\R^d\setminus(S\cup\{0\}))dt\right],
\]
which is equal to zero by definition of $S$. Thus $\sum_{t>0}\bm1_{\R^d\setminus(S\cup\{0\})}(\Delta X_t(\omega))=0$ for all $\omega$ outside some nullset $N$. In view of the convention $X_0=X_{0-}$, this is precisely what we claimed. Now, pick any $\omega\notin N$, along with $t\ge0$ such that $\Delta Z_t(\omega)\ne0$. Then the vector $u=\Delta X_t(\omega)$ is nonzero and hence lies in $S$. Let $\psi_u$ be the corresponding normalized jump counter, which satisfies
\[
\psi_u(X_t(\omega))-\psi_u(X_{t-}(\omega))=\psi_u(u)-\psi_u(0)=-1
\]
by definition. On the other hand, Lemma~\ref{L:F1}\ref{L:F1:3} yields $c_0,\ldots,c_k\in\R$ such that $\psi_u=c_0+c_1\psi_{u_1}+\cdots+c_k\psi_{u_k}$, whence
\[
\psi_u(X_t(\omega))-\psi_u(X_{t-}(\omega)) = c_1 (\Delta Y_t)_1(\omega) + \cdots + c_k(\Delta Y_t)_k(\omega).
\]
It follows that $\Delta Y_t(\omega)\ne0$, which proves~\eqref{T:JM:3_1}. To obtain the form of the characteristic $\nu^Z(dt,d\zeta)$, define $\Phi\colon\R^d\to\R^{d-k}$ by $\Phi(x)_j=T(x)_{k+j}$ for $j=1,\ldots,d-k$, so that $Z=\Phi(X)$. Then, as above, we have
\[
\nu^Z([0,t]\times C) = \int_0^t \int_{\R^d\setminus\{0\}} \bm 1_{\{\Phi(X_{s-}+\xi)-\Phi(X_{s-})\in C\}} F(X_{s-},d\xi)ds = \int_0^t F^Z(Y_{s-},C)ds
\]
for any $t\ge0$ and measurable $C\subseteq\R^{d-k}\setminus\{0\}$, where $F^Z(y,C)=F_1(y, \Phi^{-1}(\Phi(0)+C))$. That $F^Z$ maps $\R^k$ to $\mathfrak M^{d-k}_*$ follows in the same way as the corresponding statement for $F^Y$ above. This completes the proof of~\ref{T:JM:3}.

It remains to prove \ref{T:JM:4}. For each $j=1,\ldots,k$, the vector $v_j:=T(u_j)-T(0)$ is a possible jump size of $(Y,Z)$. We check that $\pi_j$ is the corresponding normalized jump counter. Certainly $\pi_j$ is affine, not identically zero, and nonnegative on $T(E)\subseteq\N^k\times\R^{d-k}$. Since $\psi_{u_j}$ is normalized, so is $\pi_j$, because
\[
\pi_j(v_j)-\pi_j(0) = \pi_j(v_j) = \psi_{u_j}(u_j)-\psi_{u_j}(0) = -1.
\]
Finally, if $\pi_j(T(x))>0$ for some $x\in E$, then $\psi_{u_j}(x)>0$, hence $x+u_j\in E$, and therefore
\[
T(x) + v_j = T(x+u_j) \in T(E).
\]
This completes the proof of \ref{T:JM:4} and of the theorem.
\end{proof}

\section{Examples and further classification} \label{S:ex}

In this section we classify all affine processes with compact state space in one and two dimensions. For an affine process $X$ with triplet $(b,c,F)$, we continue to let $S$ denote the set \eqref{eq:S} of possible jump sizes.

\paragraph{Dimension $d=1$.} In this case Theorem~\ref{T:main} yields, up to an affine transformation, two possible affine processes $X$ with compact state space: either $X$ is deterministic, or $X$ is a finite-state continuous time Markov chain. We only inspect the second possibility, which turns out to result in a simple birth--death process.

\begin{proposition}\label{p:1d}
Assume $X$ is a non-deterministic\footnote{By non-deterministic we mean that for some $x\in E$, the law of $X$ under $\P_x$ is not a point mass concentrated on one single c\`adl\`ag function.} affine process with compact state space $E\subseteq\R$. Then, up to an affine transformation, we have $E=\{0,\dots,N\}$ for some integer $N\ge1$, and there are real numbers $\alpha>0$, $\beta\geq 0$ such that
\[
F(x,\cdot) = x\alpha\delta_{-1} + (N-x)\beta\delta_1,\qquad x\in\mathbb R.
\]
Moreover, the moment generating function of $X$ is given by
\[
\E_x[e^{uX_t}] = \Phi(u,t)\Psi(u,t)^x,\qquad x\in E,
\]
for any $u\in\C$ and $t\ge0$, where we set $0^0=1$, and the functions
\begin{align*}
\Phi(u,t) &= \left(\frac{\alpha+\beta(e^u+(1-e^u)e^{-t(\alpha+\beta)})}{\alpha+\beta}\right)^{N}\,, \\
\Psi(u,t) &= 1 + \frac{(\beta+\alpha)(e^u-1)}{(\beta e^u+\alpha)e^{t(\alpha+\beta)}-\beta(e^u-1)}\,,
\end{align*}
solve the Riccati equations
\begin{align*}
\partial_t\Phi(u,t) &= N\beta \Phi(u,t) (\Psi(u,t)-1), && \Phi(u,0)=1, \\
\partial_t\Psi(u,t) &= \alpha + (\beta-\alpha)\Psi(u,t) - \beta\Psi(u,t)^2, && \Psi(u,0)=e^u.
\end{align*}
\end{proposition}

\begin{proof}
Theorem~\ref{T:main} yields, after an affine transformation, that $E\subseteq \mathbb N$ and that the coordinate projection $\pi_1\colon x\mapsto x$, which coincides with the identity since $d=1$, is a normalized jump counter. In particular $E\subseteq\{0,\ldots,N\}$ for some $N\in E$, and since $X$ is non-deterministic, $E$ must contain at least two states, so $N\ge1$. Let $u\in S$ be a vector whose normalized jump counter is $\pi_1$. Then $-1=\pi_1(u)-\pi_1(0)=u$, so \eqref{eq:psi_u step} implies $E=\{0,\ldots,N\}$. If $S=\{-1\}$ then $F$ has the claimed jump structure with $\beta=0$. Otherwise, consider $v\in S$, $v\ne-1$, and its corresponding normalized jump counter $\psi_v$. Since $\pi_1(v)-\pi_1(0)=v\ne-1$, Proposition~\ref{P:JC repr} yields $v=-u=1$. Thus $S=\{-1,1\}$ and $F$ again has the claimed structure with $\beta>0$.

We now turn to the moment generating function. Fix $u\in\R$ and $t\ge0$, and define $M_s=\Phi(u,t-s)\Psi(u,t-s)^{X_s}$, $s\in[0,t]$. A calculation using that $\Phi$ and $\Psi$ solve the Riccati equations yields that $M$ is a martingale with $M_t=e^{uX_t}$. Hence
\[
\E_x[e^{uX_t}] = \E_x[ M_t ]  = M_0 = \Phi(u,t)\Psi(u,t)^x,
\]
as claimed.
\end{proof}

Observe that, while the affine processes in Proposition~\ref{p:1d} do admit closed-form moment generating and characteristic functions, there is no exponential-affine transform formula. Indeed, classically one would write
\[
\Phi(u,t)=e^{\phi(u,t)} \qquad\text{and}\qquad  \Psi(u,t)=e^{\psi(u,t)},
\]
where $\phi$ and $\psi$ solve (generalized) Riccati equations. Here this is not possible in general, as can be seen by taking $N=1$, $\alpha=1$, and $\beta=0$. In this case $\Psi(u,t)=1+(e^u-1)e^{-t}$, which vanishes for $u={\rm i}\pi+\log(e^t-1)$.

\paragraph{Dimension $d=2$, case $k=1$.}  Up to an affine transformation Theorem~\ref{T:main} yields $E\subseteq \mathbb N^k\times\mathbb R^{2-k}$ for some $k\in\{0,1,2\}$. Again $k=0$ means that $X$ is deterministic, so we ignore this case. If $k=1$, then $X=(Y,Z)$ where the first component $Y$ is itself a one-dimensional non-deterministic affine process, and therefore of the form described in Proposition~\ref{p:1d} with state space $\{0,\ldots,N\}$ for some nonzero $N\in\N$ and jump kernel $F^Y(y,\cdot) = y\alpha\delta_{-1} + (N-y)\beta\delta_1$ for some parameters $\alpha>0$ and $\beta\ge0$. The second component of $Z$ may or may not jump, depending on whether $\beta$ is zero or not. The following proposition gives the precise statement.

\begin{proposition}\label{p:k=1_new}
Consider the affine process $X=(Y,Z)$ just described, with $F^Y(y,\cdot) = y\alpha\delta_{-1} + (N-y)\beta\delta_1$ for some $\alpha>0$ and $\beta\ge0$. If $\beta>0$, then up to an affine transformation of $X$, $Z$ is continuous and we have
\[
E=\{0,\ldots,N\}\times E^Z
\]
for some compact subset $E^Z\subseteq\R$. If instead $\beta=0$, then $Z$ need not be continuous, but can jump at most $y$ times $\P_{(y,z)}$-a.s.~for each $(y,z)\in E$.
\end{proposition}

\begin{proof} 
Suppose $\beta>0$. There are compact subsets $F_0,\ldots,F_N\subseteq\R$ such that $E = \bigcup_{j=0}^N\{j\}\times F_j$. Let $f_j=\min F_j$ for $j=0, \dots, N$, and define $\phi(y,z)=z-f_0-y(f_1-f_0)$ for any $(y,z)\in\R^2$. Then the map $(y,z)\mapsto(y,\phi(y,z))$ is affine and invertible, and the process $(Y,\phi(Y,Z))$ is affine with state space $\bigcup_{j=0}^N\{j\}\times G_j$, where $G_j=F_j-f_0+j(f_1-f_0)$. In particular, $\min G_0=\min G_1=0$. Moreover, this process still satisfies the properties of Theorem~\ref{T:main} (with $Z$ replaced by $\phi(Y,Z)$). Therefore, we suppose already from the outset that
\[
\text{$\min F_0=\min F_1 = 0$ and the map $T$ in Theorem~\ref{T:main} is the identity.}
\]
This initial transformation is illustrated in Figure~\ref{fig1}.

Now, pick $u=(u_1,u_2)\in S$. Since $Y$ jumps by $\pm1$, we have $u_1\in\{-1,0,1\}$. But due to \eqref{T:JM:3_1}, we cannot have $u_1=0$ and $u_2\ne0$, so in fact $u_1\in\{-1,1\}$. Suppose $u_1=1$ and let $\psi_u$ be the associated normalized jump counter. Since $Y$ does not jump in direction $u_1=1$ when $Y_t=N$, \eqref{T:JM:3_1} implies that $Z$ cannot jump when $Y_t=N$. Thus $\{N\}\times\R\subseteq\ker\psi_u$, which together with the normalization condition uniquely determines $\psi_u$. Similarly, for $v\in S$ with $v_1=-1$ we have $\{0\}\times\R\subseteq\ker\psi_v$, and this uniquely determines $\psi_v$. It follows that $\psi_u$ and $\psi_v$ are the only normalized jump counters.

Next, since $\psi_u(0,z)>0$ for all $z\in F_0$, we have $F_0+u\subseteq F_1$, and similarly $F_1+v\subseteq F_0$. Since $\min F_0=\min F_1=0$ this yields $u_2\ge0$ and $v_2\ge0$, and hence $u+v=(0,\lambda)$ for some $\lambda\ge0$. But since $F_0 + (u+v) \subseteq F_0$, and hence $F_0 + (0,n\lambda)\subseteq F_0$ for all $n\in\N$, compactness of $F_0$ forces $\lambda=0$. Thus $u=(1,0)$, $v=(-1,0)$, $S=\{u,v\}$, and it follows that $Z$ is continuous. It also follows that $F_j=F_0$ for all $j=1,\ldots,N$, since otherwise a jump in direction $u$ or $v$ would for some point lead out of the state space. This completes the proof of the case $\beta>0$.

Suppose $\beta=0$. Then $Y$ only has downward jumps of unit size and stops when it reaches zero. Thus $Y$ jumps exactly $y$ times, $\P_{(y,z)}$-a.s., for each $(y,z)\in E$. Due to \eqref{T:JM:3_1}, $Z$ thus jumps at most $y$ times, $\P_{(y,z)}$-a.s.
\end{proof}

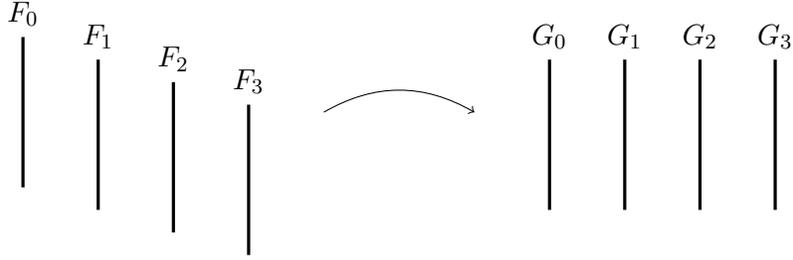
\begin{figure}
\begin{center}
\begin{tikzpicture}
\path[->] (4,1) edge [bend left] node[above] {} (6,1);
\draw[very thick,-] (0,0) -- (0,2) node[above]{$F_0$};
\draw[very thick,-] (1,-0.3) -- (1,1.7) node[above]{$F_1$};
\draw[very thick,-] (2,-0.6) -- (2,1.4) node[above]{$F_2$};
\draw[very thick,-] (3,-0.9) -- (3,1.1) node[above]{$F_3$};
\draw[very thick,-] (7,-0.3) -- (7,1.7) node[above]{$G_0$};
\draw[very thick,-] (8,-0.3) -- (8,1.7) node[above]{$G_1$};
\draw[very thick,-] (9,-0.3) -- (9,1.7) node[above]{$G_2$};
\draw[very thick,-] (10,-0.3) -- (10,1.7) node[above]{$G_3$};
\end{tikzpicture}
\caption{The map $(y,z)\mapsto(y,\phi(y,z))$ in the proof of Proposition~\ref{p:k=1_new}.}
\label{fig1}
\end{center}
\end{figure}

We now give two examples corresponding to the two cases $\beta>0$ and $\beta=0$ in Proposition~\ref{p:k=1_new}. The first example shows that while $Z$ is continuous, its drift may still depend on the jump component $Y$. The second example shows that for $\beta=0$, the layers $F_0,\dots,F_N$ need not be equal, and the jumps of $Z$ cannot be eliminated by applying invertible affine transformations.

\begin{example}
Let $N\in\mathbb N$ be nonzero and set $E=\{0,\dots, N\}\times [0,1]$. Define 
\[
F((y,z),\cdot) = y\delta_{(-1,0)} + (N-y)\delta_{(1,0)} \qquad\text{and}\qquad b(y,z) = (0,y/N-z).
\]
For each $x=(y,z)\in E$, the law $\P_x$ of the canonical process $X=(Y,Z)$ is specified as follows. First, $Y$ is the continuous time Markov chain with state space $\{0,\dots,N\}$, jump intensity measure
\[
F^Y(y,\cdot) = y\delta_{-1} + (N-y)\delta_{1},
\]
and initial condition $Y_0=y$. Next, $Z$ is the solution of the equation
\[
dZ_t = b(Y_t,Z_t)dt
\]
with initial condition $Z_0=z$. Then $X=(Y,Z)$ is affine as in Proposition~\ref{p:k=1_new} with $\beta>0$.
\end{example}

\begin{example} \label{ex:k1}
Let $N\in\mathbb N$ be nonzero and set $E=\bigcup_{j=0}^N\{j\}\times [0,N-j]$. Define
\[
F((y,z),\cdot) = y(\delta_{-1}\otimes\lambda) \qquad\text{and}\qquad    b(y,z) = (0,-z),
\]
where $\lambda$ denotes Lebesgue measure on $[0,1]$. For each $x=(y,z)\in E$, the law $\P_x$ of the canonical process $X=(Y,Z)$ is specified as follows. Before the first jump and starting from $z\in[0,N-y]$, $Z$ satisfies the equation $dZ_t = -Z_tdt$. When the first jump occurs, which happens with intensity $y$, $Y$ jumps from $y$ to $y-1$ and $Z$ makes a positive jump of uniformly distributed size. The process lands in the layer $\{y-1\}\times[0,N-y+1]$, and $Z$ continues to perform its downward motion until the next jump, which happens with intensity $y-1$, and so on. The resulting process is affine as in Proposition~\ref{p:k=1_new} with $\beta=0$. Moreover, the set $S$ of possible jump sizes affinely spans $\R^2$, so it is clear that no invertible affine transformation can eliminate the jumps in one of the components. See Figure~\ref{fig2} for an illustration.
\end{example}

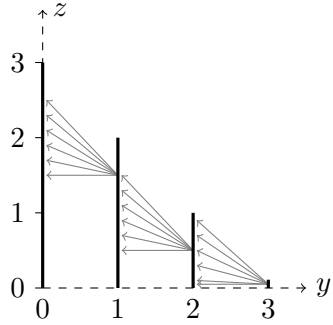
\begin{figure}
\begin{center}
\begin{tikzpicture}
\draw[thin,dashed,->] (0,0) -- (0,3.7) node[right] {$z$};
\draw[thin,dashed,->] (0,0) -- (3.5,0) node[right] {$y$};
\path[gray,->] (1,1.5) edge (0.05,2.5);
\path[gray,->] (1,1.5) edge (0.05,2.3);
\path[gray,->] (1,1.5) edge (0.05,2.1);
\path[gray,->] (1,1.5) edge (0.05,1.9);
\path[gray,->] (1,1.5) edge (0.05,1.7);
\path[gray,->] (1,1.5) edge (0.05,1.5);
\path[gray,->] (2,0.5) edge (1.05,1.5);
\path[gray,->] (2,0.5) edge (1.05,1.3);
\path[gray,->] (2,0.5) edge (1.05,1.1);
\path[gray,->] (2,0.5) edge (1.05,0.9);
\path[gray,->] (2,0.5) edge (1.05,0.7);
\path[gray,->] (2,0.5) edge (1.05,0.5);
\path[gray,->] (3,0.05) edge (2.05,0.9);
\path[gray,->] (3,0.05) edge (2.05,0.7);
\path[gray,->] (3,0.05) edge (2.05,0.5);
\path[gray,->] (3,0.05) edge (2.05,0.3);
\path[gray,->] (3,0.05) edge (2.05,0.1);
\path[gray,->] (3,0.05) edge (2.05,0.05);
\draw[thin,-] (0,0) -- (-0.1,0) node[left] {$0$};
\draw[thin,-] (0,1) -- (-0.1,1) node[left] {$1$};
\draw[thin,-] (0,2) -- (-0.1,2) node[left] {$2$};
\draw[thin,-] (0,3) -- (-0.1,3) node[left] {$3$};
\draw[very thick,-] (0,0) node[below]{$0$} -- (0,3) ;
\draw[very thick,-] (1,0) node[below]{$1$} -- (1,2) ;
\draw[very thick,-] (2,0) node[below]{$2$} -- (2,1) ;
\draw[very thick,-] (3,0) node[below]{$3$} -- (3,0.11) ;
\end{tikzpicture}
\caption{The process in Example~\ref{ex:k1} jumps from any layer $\{j\}\times[0,N-j]$ to the next layer to the left. The vertical component of the jump size is standard uniform. Within a layer, the process performs a downward linear drift motion.}
\label{fig2}
\end{center}
\end{figure}

\paragraph{Dimension $d=2$, case $k=2$.} It remains to consider the case where, up to an affine transformation, we have $E\subseteq\N^2$. That is, we assume $X=T(X)=Y$ in Theorem~\ref{T:main}, so that $k=2$ and the coordinate projections $\pi_j\colon x\mapsto x_j$, $j=1,2$, are normalized jump counters. The following result classifies this situation, and shows that there are three possibilities.

\begin{theorem}\label{t:jumps in 2d}
Assume $X=(X_1,X_2)$ is affine with state space $E\subseteq\N^2$, and that $\pi_1, \pi_2$ are normalized jump counters. Then $X$ is of one of the following three types:
\begin{enumerate}
  \item\label{t:jumps in 2d:1} Up to a further affine transformation, $E$ has a layer structure in the sense that $E_0=\{x\in E\colon x_2=0\}$ is stochastically invariant and $\pi_2(u)\leq 0$ for every $u\in S$ so that there are no upward jumps. Thus there are at most $N=\max\{x_2\colon x\in E\}$ downward jumps, after which the process arrives in $E_0$ and stays there. Moreover, we have
   $$ \{(-1,0)\}\subseteq S \subseteq \{(-1,0),(1,0)\}\cup\{(K,-1):K\in\mathbb N\}. $$
  \item\label{t:jumps in 2d:2} $X_1$ and $X_2$ are independent affine processes, and
   $$ \{(-1,0),(0,-1)\} \subseteq S \subseteq \{(-1,0),(0,-1),(1,0),(0,1)\}. $$
     All normalized jump counters apart from $\pi_1$ and $\pi_2$ are of the form $N-\pi_1$ or $K-\pi_2$ for some $N,K\in\mathbb N$.
  \item\label{t:jumps in 2d:3} The set of normalized jump counters is either $\{\pi_1,\pi_2\}$ or $\{\pi_0,\pi_1,\pi_2\}$, where $\pi_0=N-\pi_1-\pi_2$ for some $N\in\mathbb N$. The set of possible jump sizes $S$ satisfies 
  $$S\subseteq\{(-1,0),(-1,1),(0,-1),(0,1),(1,-1),(1,0)\}.$$
\end{enumerate}
\end{theorem}

Before giving the proof of Theorem~\ref{t:jumps in 2d} we give three examples to illustrate the three possibilities. First, Case~\ref{t:jumps in 2d:2} is easily constructed by taking two independent one-dimensional affine jump-processes as in Proposition~\ref{p:1d}. Next, we consider an example of Case~\ref{t:jumps in 2d:1}.

\begin{example} \label{ex:d2k2:1}
  We construct an affine process on a state space with three layers $\{x_2=2\}$, $\{x_2=1\}$, and $\{x_2=0\}$. The situation is illustrated in Figure~\ref{fig3}. Let
  \begin{align*} 
    E&=\{(0,2),(0,1),(1,1),(2,1),(3,1),(0,0),(1,0),(2,0),(3,0),(4,0),(5,0),(6,0),(7,0)\}, \\
    S&=\{(-1,0),(0,-1),(2,-1),(3,-1)\}, \\
    \psi_u&= \begin{cases} \pi_1, & u=(-1,0), \\ \pi_2, & \text{any other $u\in S$,} \end{cases} \\
    F(x,\cdot)&= \sum_{u\in S} \psi_u(x)\delta_u, \quad x\in\R^d.
  \end{align*}
 Then there is an affine process with triplet $(0,0,F)$ and state space $E$. Note that one could add points $(n,0)$, $n=8,9,\ldots$, to the state space and still obtain an affine process. Note also that the process eventually gets trapped in $E_0=\{x_2=0\}$, and in this particular example even in $(0,0)$.
\end{example}

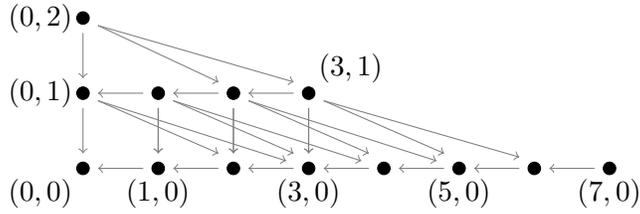
\begin{figure}
\begin{center}
\begin{tikzpicture}
\draw[black,fill=black] (0,2) circle (.5ex) node[left] {$(0,2)$};
\draw[black,fill=black] (0,1) circle (.5ex) node[left] {$(0,1)$};
\draw[black,fill=black] (1,1) circle (.5ex);
\draw[black,fill=black] (2,1) circle (.5ex);
\draw[black,fill=black] (3,1) circle (.5ex) node[above right] {$(3,1)$};
\draw[black,fill=black] (0,0) circle (.5ex) node[below left] {$(0,0)$};
\draw[black,fill=black] (1,0) circle (.5ex) node[below] {$(1,0)$};
\draw[black,fill=black] (2,0) circle (.5ex);
\draw[black,fill=black] (3,0) circle (.5ex) node[below] {$(3,0)$};
\draw[black,fill=black] (4,0) circle (.5ex);
\draw[black,fill=black] (5,0) circle (.5ex) node[below] {$(5,0)$};
\draw[black,fill=black] (6,0) circle (.5ex);
\draw[black,fill=black] (7,0) circle (.5ex) node[below] {$(7,0)$};

\path[gray,->] (0,2-.2) edge (0,1+.2);
\path[gray,->] (0,1-.2) edge (0,0+.2);
\path[gray,->] (1,1-.2) edge (1,0+.2);
\path[gray,->] (2,1-.2) edge (2,0+.2);
\path[gray,->] (3,1-.2) edge (3,0+.2);
\path[gray,->] (1,1-.2) edge (1,0+.2);
\path[gray,->] (2,1-.2) edge (2,0+.2);
\path[gray,->] (3,1-.2) edge (3,0+.2);

\path[gray,->] (1-.2,1) edge (0+.2,1);
\path[gray,->] (2-.2,1) edge (1+.2,1);
\path[gray,->] (3-.2,1) edge (2+.2,1);

\path[gray,->] (1-.2,0) edge (0+.2,0);
\path[gray,->] (2-.2,0) edge (1+.2,0);
\path[gray,->] (3-.2,0) edge (2+.2,0);
\path[gray,->] (4-.2,0) edge (3+.2,0);
\path[gray,->] (5-.2,0) edge (4+.2,0);
\path[gray,->] (6-.2,0) edge (5+.2,0);
\path[gray,->] (7-.2,0) edge (6+.2,0);

\path[gray,->] (0+.2,2-.1) edge (2-.2,1+.15);
\path[gray,->] (0+.2,2-.1) edge (3-.2,1+.15);

\path[gray,->] (3+.2,1-.1) edge (5-.2,0+.2);
\path[gray,->] (3+.2,1-.1) edge (6-.2,0+.15);

\path[gray,->] (2+.2,1-.1) edge (4-.2,0+.2);
\path[gray,->] (2+.2,1-.1) edge (5-.3,0+.1);

\path[gray,->] (1+.2,1-.1) edge (3-.2,0+.2);
\path[gray,->] (1+.2,1-.1) edge (4-.3,0+.1);

\path[gray,->] (0+.2,1-.1) edge (2-.2,0+.2);
\path[gray,->] (0+.2,1-.1) edge (3-.3,0+.1);
\end{tikzpicture}
\caption{The process in Example~\ref{ex:d2k2:1}.}
\label{fig3}
\end{center}
\end{figure}

Finally, we give an example of Case~\ref{t:jumps in 2d:3}. In fact, we provide a more general example of a $d$-dimensional affine process with compact and irreducible state space and no autonomous components. We say that a $d$-dimensional affine process $X$ has {\em no autonomous components} if there is no invertible affine map $T\colon\R^d\to\R^d$ such that $(T(X)_1,\ldots,T(X)_k)$ is itself an affine process for some $k\le d-1$.

\begin{example}\label{ex:ES}
  Let $d,N\in\mathbb N$ with $d,N\geq 1$ and define 
  \begin{align*}
    E&=\Big\{x\in\mathbb N^d\colon \sum_{j=1}^d x_j\leq N\Big\}, \\
    \pi_0 &= N - \sum_{j=1}^d \pi_j, \\
    S &= \{ e_j-e_k\colon j,k=0,\dots,d,j\neq k\}, \\
  \end{align*}
  where as usual $\pi_j\colon x\mapsto x_j$ are the coordinate projections, $e_0=0$, and $e_j$ is the $j$th canonical unit vector for $j=1,\ldots,d$. The state space $E$ consists of the points of the solid simplex with integer coordinates. For each element $u=e_j-e_k\in S$, define also $\psi_u=\pi_k$. Let $\lambda\colon S\to (0,\infty)$ be arbitrary, and define
 $$ F(x,\cdot) = \sum_{u\in S}  \lambda(u)\psi_u(x)\delta_u , \quad x\in\R^d. $$
Then $F(x,\cdot)$ is a finite measure concentrated on $E-x$ for any $x\in E$, and $x\mapsto F(x,\cdot)$ is affine. It is then straightforward to construct an affine process $X$ with triplet $(0,0,F)$ and state space $E$. The set $S$ is the set of its possibly jump sizes, and $\psi_u$ is the normalized jump-counter corresponding to $u\in S$. Moreover, $X$ has compact and irreducible state space and no autonomous components. With the special choices $d=2$ and $N=3$ we obtain the situation illustrated in Figure~\ref{fig4}, and the set of possible jump sizes is
\begin{equation} \label{eq:ES1}
S = \{(-1,0),(-1,1),(0,-1),(0,1),(1,-1),(1,0)\}.
\end{equation}
If we allow $\lambda(u)=0$ additionally for some $u$, then the corresponding jump does not occur, and the set of possible jump sizes becomes a subset of the one given in~\eqref{eq:ES1}.
\end{example}

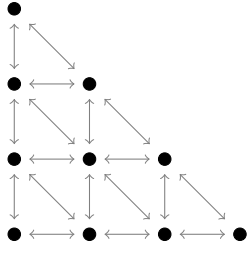
\begin{figure}
\begin{center}
\begin{tikzpicture}
\draw[black,fill=black] (0,3) circle (.5ex);
\draw[black,fill=black] (0,2) circle (.5ex);
\draw[black,fill=black] (0,1) circle (.5ex);
\draw[black,fill=black] (0,0) circle (.5ex);
\draw[black,fill=black] (1,2) circle (.5ex);
\draw[black,fill=black] (1,1) circle (.5ex);
\draw[black,fill=black] (1,0) circle (.5ex);
\draw[black,fill=black] (2,1) circle (.5ex);
\draw[black,fill=black] (2,0) circle (.5ex);
\draw[black,fill=black] (3,0) circle (.5ex);

\path[gray,<->] (0,3-.2) edge (0,2+.2);
\path[gray,<->] (0,2-.2) edge (0,1+.2);
\path[gray,<->] (0,1-.2) edge (0,0+.2);
\path[gray,<->] (1,2-.2) edge (1,1+.2);
\path[gray,<->] (1,1-.2) edge (1,0+.2);
\path[gray,<->] (2,1-.2) edge (2,0+.2);

\path[gray,<->] (3-.2,0) edge (2+.2,0);
\path[gray,<->] (2-.2,0) edge (1+.2,0);
\path[gray,<->] (1-.2,0) edge (0+.2,0);
\path[gray,<->] (2-.2,1) edge (1+.2,1);
\path[gray,<->] (1-.2,1) edge (0+.2,1);
\path[gray,<->] (1-.2,2) edge (0+.2,2);

\path[gray,<->] (0+.2,3-.2) edge (1-.2,2+.2);
\path[gray,<->] (0+.2,2-.2) edge (1-.2,1+.2);
\path[gray,<->] (0+.2,1-.2) edge (1-.2,0+.2);
\path[gray,<->] (1+.2,2-.2) edge (2-.2,1+.2);
\path[gray,<->] (1+.2,1-.2) edge (2-.2,0+.2);
\path[gray,<->] (2+.2,1-.2) edge (3-.2,0+.2);
\end{tikzpicture}
\caption{The process in Example~\ref{ex:ES} for $d=2$ and $N=3$.}
\label{fig4}
\end{center}
\end{figure}

In fact, we believe that Example \ref{ex:ES} is the only possibility of this type:

\begin{conjecture}
Let $d\ge2$ and let $X$ be a $d$-dimensional affine process with finite and irreducible state space and no autonomous components. Then there is an invertible affine transformation $T\colon\R^d\to\R^d$ such that $T(X)$ coincides with the process constructed in Example~\ref{ex:ES}, with the only difference that the function $\lambda$ may be $[0,\infty)$-valued instead of $(0,\infty)$-valued.
\end{conjecture}

After this series of examples we now turn to the proof of Theorem~\ref{t:jumps in 2d}. We first show that ``large jumps'' imply the layer structure in Theorem~\ref{t:jumps in 2d}\ref{t:jumps in 2d:1}; see Lemma \ref{l:large jumps} below. Thereafter, we assume that there are no ``large jumps'', which has significant further implications.

\begin{lemma}\label{l:large jumps}
Assume that there is $u\in S$ such that $\max\{|u_1|,|u_2|\}\ge2$. Then any $v\in S$ satisfies $\psi_u(v)-\psi_u(0) \in \{-1,0\}$. Moreover, for any $v,v'\in S$ such that $\psi_u(v)-\psi_u(0) = \psi_u(v')-\psi_u(0) = 0$ we have $v'\in\{v,-v\}$. In particular, after applying the affine transformation $T=(\pi_1,\psi_u)$ if $|u_1|\ge2$, or $T=(\pi_2,\psi_u)$ otherwise, we obtain case~\ref{t:jumps in 2d:1} in Theorem~\ref{t:jumps in 2d}.
\end{lemma}

\begin{proof}
We assume without loss of generality that $|u_1|\ge2$. Let $w\in S$ be a jump size whose normalized jump counter is $\psi_w=\pi_1$. Then Proposition~\ref{P:JC repr} with $u$ and $w$ yields
\begin{equation} \label{eq:large jumps 1}
\psi_u(w) - \psi_u(0) = 0 \quad\text{and}\quad \pi_1(u)=u_1\ge2.
\end{equation}
Define $T=(\pi_1,\psi_u)$ and consider the transformed process $T(X)$. Then $a = Tw - T0 \in T(S) - T(0)$ is in the set of possible jump sizes of $T(X)$ and its normalized jump counter is $\pi_1$, and $b= Tu-T0 \in T(S) - T(0)$ is in the set of possible jump sizes of $T(X)$ and its normalized jump counter is $\pi_2$.
Moreover, we have 
\begin{align*}
   &\pi_1(a) = \pi_1(w) = -1,&& \pi_2(a) = \psi_u(w)-\psi_u(0) = 0, \\
   &\pi_1(b) = \pi_1(u) \geq 2,&& \pi_2(b) = \psi_u(u)-\psi_u(0) = -1.
\end{align*}
That is $a=(-1,0)$, its normalized jump counter is $\pi_1$ and $b=(b_1,-1)$ with $b_1\geq 2$ and normalized jump counter $\pi_2$.

Pick any $c\in T(S)-T(0)$. We claim that $\pi_2(c) \in \{0,-1\}$, and assume for contradiction that $\pi_2(c) \notin \{0,-1\}$. Proposition \ref{P:JC repr} yields that $\pi_2(c) \geq 1$. If $\pi_1(c) \leq -1$, then $\pi_1$ would be the normalized jump counter of $c$ and, hence, $c_1=-1$. Proposition \ref{P:JC repr} applied to $b$ and $c$ would yield $\pi_2(c) = 0$. Thus $c_1 = \pi_1(c)\geq 0$. Thus, $c \neq -b$ and Proposition \ref{P:JC repr} applied to $b$ and $c$ yields $\psi_c(b)-\psi_c(0) = 0$ because $\pi_2(c)\geq 1$. Since $a,b$ are linearly independent and $\psi_c(b)-\psi_c(0) = 0$ we get $\psi_c(a)-\psi_c(0) \neq 0$. Since $a\neq -c$ and $\pi_1(c)\geq 0$ we get from Proposition \ref{P:JC repr} that $\psi_c(a)-\psi_c(0)\geq 1$ and $\pi_1(c) = 0$. Thus $c=(0,c_2)$.

We have $\psi_c = K - \alpha \pi_1 - \beta \pi_2$ for some $K,\alpha,\beta\in \mathbb R$. We have $1 \leq \psi_c(a)-\psi_c(0) = \alpha \in\mathbb N$. We also have $0 = \psi_c(b) - \psi_c(0) = -\alpha b_1 + \beta$ and, hence, $\beta = \alpha b_1 \geq 2$ and $\beta\in\mathbb N$. However, $-1 = \psi_c(c)-\psi_c(0) = -\beta c_2 \leq -2$. A contradiction.
\end{proof}

We now inspect the possibility that there are exactly two normalized jump counters.

\begin{lemma}\label{l:two jump kernels}
  Assume that $\Vert u\Vert_\infty=1$ for every $u\in S$, and assume that $\{\psi_u\colon u\in S\}$ has two elements, namely $\pi_1$ and $\pi_2$. Then, there are $a,b,c,d\in \mathbb R_+$ such that 
   $$ F(x,\cdot) = a\,\pi_1(x)\delta_{(-1,0)} + b\,\pi_1(x)\delta_{(-1,1)} + c\,\pi_2(x)\delta_{(0,-1)} + d\,\pi_2(x)\delta_{(1,-1)},\quad x\in E.$$
  In particular, we have case~\ref{t:jumps in 2d:3} in Theorem~\ref{t:jumps in 2d}.
\end{lemma}
\begin{proof}
Let $u\in S$. Then $\pi_1(u)=-1$ or $\pi_2(u)=-1$. Suppose $\pi_1(u) = -1$, i.e.\ $u=(-1,u_2)$. Then Proposition \ref{P:JC repr} yields $\pi_2(u)\in\{0,1\}$. The case $\pi_2(u)=-1$ is similar, and we deduce that $S$ can only contain $(-1,0)$ and $(-1,1)$ (corresponding to $\pi_1$) and $(0,-1)$ and $(1,-1)$ (corresponding to $\pi_2$).
\end{proof}

Next, we consider no ``large jumps'' and at least three normalized jump counters.

\begin{lemma}\label{l:more jump kernels}
Assume that $\max\{|u_1|,|u_2|\}=1$ for every $u\in S$, and assume there exists $w\in S$ such that $\psi_w\notin\{\pi_1,\pi_2\}$. Then, after some affine transformation, either we have case~\ref{t:jumps in 2d:1} in Theorem~\ref{t:jumps in 2d}, or we have
\[ S\subseteq \{(-1,0),(-1,1),(0,-1),(0,1),(1,-1),(1,0)\}. \]
\end{lemma}

\begin{proof}
The same argument as in the proof of Lemma \ref{l:two jump kernels} yields $u\in \{(-1,0),(-1,1)\}$ for any $u\in S$ with $\psi_u=\pi_1$, and $u\in \{(0,-1),(1,-1)\}$ for any $u\in S$ with $\psi_u=\pi_2$. Next, choose any $w\in S$ with $\psi_w\notin\{\pi_1,\pi_2\}$. Proposition~\ref{P:JC repr} yields $\pi_1(w)\ne-1$. Since also $\pi_1(w)\le1$ by the assumption on the jump sizes, we get $\pi_1(w)\in\{0,1\}$. The same argument yields $\pi_2(w)\in\{0,1\}$. Therefore, we have $w\in\{(0,1),(1,0),(1,1)\}$.

If $(1,1)\notin S$, then $S\subseteq \{(-1,0),(-1,1),(0,-1),(0,1),(1,-1),(1,0)\}$, as claimed. Suppose instead that $(1,1)\in S$. Let $w=(1,1)$ and observe that $\psi_w \notin\{\pi_1,\pi_2\}$. Let $u,v\in S$ have normalized jump counters $\pi_1,\pi_2$, respectively.  We have seen above that $u\ne-w$, so Proposition~\ref{P:JC repr} yields $\psi_w(u)-\psi_w(0)=0$. Similarly, $\psi_w(v)-\psi_w(0)=0$. In particular, $u$ and $v$ both lie in $\ker(\psi_w-\psi_w(0))$, which implies that they are proportional and therefore given by $u=-v=(-1,1)$. Consequently, $S\subseteq \{u,v,w,(1,0),(0,1)\}$. Now, since $(-1,1)$ spans $\ker(\psi_w-\psi_w(0))$, we deduce that $\psi_w=N-(\pi_1+\pi_2)/2$ for some $N\in\N$. Assume for contradiction that $(1,0) \in S$. Then $\psi_{(1,0)} = \psi_w$. Moreover, there is $x \in E$ such that $\psi_w(x) = 1$. We thus have $0 = \psi_w( x + (1,0) ) = 1/2$, which is a contradiction. Hence $(1,0)\notin S$, and we similarly deduce $(0,1) \notin S$. Thus $S=  \{u,v,w\}$. Now define
\[
T = (\pi_1,\psi_w-\psi_w(0)).
\]
Then $Tu = (-1,0)$, $Tw = (1,-1)$ and $Tv = (1,0)$. Thus, we have (i) in Theorem \ref{t:jumps in 2d} after applying the affine transformation $T$.
\end{proof}

We are now ready to give the proof of Theorem~\ref{t:jumps in 2d}.

\begin{proof}[Proof of Theorem \ref{t:jumps in 2d}]
We assume that neither case~\ref{t:jumps in 2d:1} nor case~\ref{t:jumps in 2d:3} in Theorem~\ref{t:jumps in 2d} holds, and prove that case~\ref{t:jumps in 2d:2} must then hold. Lemma \ref{l:large jumps} yields $\|y\|_\infty= 1$ for every $y\in S$, and Lemma \ref{l:two jump kernels} yields $\psi_z\notin\{\pi_1,\pi_2\}$ for some $z\in S$. Lemma \ref{l:more jump kernels} then yields that
 \[
 S \subseteq \{(-1,0),(-1,1),(0,-1),(0,1),(1,-1),(1,0)\}.
 \]
In particular, $z\in \{(0,1),(1,0)\}$ and we consider the case $z=(1,0)$ (the case $z=(0,1)$ is analogous). Let $x,y\in S$ such that $\pi_1(x) = -1$ and $\pi_2(y) = -1$. Then $x\in \{(-1,0),(-1,1)\}$ and $y\in \{(0,-1),(1,-1)\}$. Since Theorem~\ref{t:jumps in 2d}\ref{t:jumps in 2d:1} does not hold, there is $w\in S$ such that $\pi_2(w) = 1$. Hence, $w\in\{(0,1),(-1,1)\}$.

Being an affine function, $\psi_z$ can be written $\psi_z = N -b\pi_1-a\pi_2$ for some $N,b,a\in\R$. We have $1 = \psi_z(0)-\psi_z(z) = b$ and hence $\psi_z= N-\pi_1-a\pi_2$. Moreover, since $\pi_2(z) = 0$ we obtain $-y_1+a = \psi_z(y)-\psi_z(0) \in \mathbb N$ which yields $a\in\mathbb N$. Since $N-u_1-au_2 = \psi_z(u) \in \N$ for any $u\in E$, we also have $N\in\N$.
 
Observe that $-1 \leq \psi_z(w)-\psi_z(0) \in \{-a,1-a\}$, so that $a \in \{0,1,2\}$. Assume for contradiction $a=2$. Then $w=(-1,1)$ and $\psi_z(w)-\psi_z(0) = -1$, which by Proposition~\ref{P:JC repr} yields $\psi_w=\psi_z$. Since $\pi_1(w)=-1$, Proposition \ref{P:JC repr} yields $\psi_w=\pi_1$, a contradiction. Thus $a\in \{0,1\}$.

Assume for contradiction $a=1$. In this case, $\psi_z = N-\pi_1-\pi_2=\pi_0$. For any $v\in S$ we then have $-1 \in \{\pi_1(v),\pi_2(v),\pi_0(v)-\pi_0(0)\}$ and, hence, $\psi_v\in \{\pi_0,\pi_1,\pi_2\}$. We thus have case~\ref{t:jumps in 2d:3}. A contradiction.

Thus $a=0$. Then $\psi_z = N-\pi_1$ and Proposition \ref{P:JC repr} yields $x=-z = (-1,0)$. Moreover, $\pi_1(w)\neq -1$ because otherwise $w=-z=(-1,0)$. This gives $w=(0,1)$. By the same argument as above, we obtain the representation $\psi_w = K - c\pi_1 - \pi_2$ for some $K\in\mathbb N$ and $c\in\{0,1\}$. However, $\psi_w(z)-\psi_w(0) = -c$ implies $c=0$ because $w\neq -z$. Thus we have $\psi_w = K-\pi_2$. Now, for any $v\in S$ we have $1 \in \{\psi_x(v),\psi_y(v),\psi_z(v)-\psi_z(0),\psi_w(v)-\psi_w(0)\}$ which yields $v\in\{x,y,z,w\}$. Thus $S=\{(-1,0),(0,-1),(1,0),(0,1)\}$ and we have case~\ref{t:jumps in 2d:2}.
\end{proof}

\appendix

\section{Some notions from convex analysis} \label{APP:convex}

Fix $n\ge0$ and let $K\subseteq\R^n$ be a closed convex set. We briefly review some notions and results from convex analysis. For further information and details, see \cite{rockafellar.70}.

\begin{itemize}
\item The {\em normal cone} of $K$ at $\overline x\in K$ is the closed convex cone
\[
N_K(\overline x) = \left\{u\in\R^n\colon \langle u,x-\overline x\rangle \le 0 \text{ for all } x\in K\right\}.
\]
Note in particular that if $\overline x\in\interior K$, then $N_K(\overline x)=\{0\}$.

\item A {\em supporting half-space} of $K$ at $\overline x\in K$ is a closed half-space which contains $K$ and whose boundary contains $\overline x$. A {\em supporting hyperplane} of $K$ at $\overline x$ is a hyperplane which is the boundary of a supporting half-space of $K$ at $\overline x$. Any supporting hyperplane $T$ of $K$ at $\overline x$ is of the form
\[
T = \{\overline x + t \colon t\in\R^n \text{ and } \langle u,t\rangle = 0\}
\]
for some $u\in N_K(\overline x)$.

\item A map $c:\R^n\to\S^n$ is said to be {\em parallel to $K$} if
\begin{equation} \label{eq:parallel def}
\text{$c(x)u=0$ for all $x\in K$ and $u\in N_K(x)$.}
\end{equation}
\end{itemize}

\begin{lemma} \label{L:K hat}
Let $V\subset\R^n$ be a linear subspace, and let $\pi:\R^n\to V$ be the orthogonal projection onto $V$. Then, for any $\overline y\in V\cap K$,
\[
N^V_{V\cap K}(\overline y) = \pi( N_K(\overline y) )
\]
where $N^V_{V\cap K}(\overline y)$ is the normal cone of $V\cap K$ in $V$.
\end{lemma}

\begin{proof}
First, let $u\in N_K(\overline y)$. Then $\langle \pi(u),y-\overline y\rangle = \langle u,\pi(y-\overline y)\rangle = \langle u,y-\overline y\rangle \le 0$ holds for any $y\in V\cap K$. Thus $\pi(u)\in N_{V\cap K}(\overline y)$.

Conversely, let $v\in N_{V\cap K}(\overline y)$. It suffices to consider the case $v\ne 0$. In this case the set
\[
T_v = \{ \overline y + t \colon t\in V \text{ and } \langle v,t\rangle = 0\}
\]
is a supporting hyperplane in $V$ of $V\cap K$ at $\overline y$. Since $T_v$ is disjoint from the interior of $K$, there exists a supporting hyperplane $T$ in $\R^n$ of $K$ at $\overline y$ such that $T_v\subseteq T$. Being a supporting hyperplane, $T$ is given by
\[
T = \{ \overline y + s \colon s\in\R^n \text{ and } \langle u,s\rangle = 0\}
\]
for some $u\in N_K(\overline y)$. Then, for any $t_1,t_2\in T_v\subseteq T$,
\[
\langle\pi(u),t_1-t_2\rangle = \langle u,\pi(t_1-t_2)\rangle = \langle u,t_1-t_2\rangle = 0.
\]
Thus $\pi(u)=\lambda v$ for some constant $\lambda$. Since both $u$ and $v$ are outward-pointing from $K$, one sees that $\lambda>0$. Thus $v\in \pi(N_K(\overline y))$ as claimed.
\end{proof}

\begin{lemma} \label{L:x normal at x}
Suppose that $n\ge 1$ and that $K$ is compact with $0\in\interior K$. Then there exists some $x\in K\setminus\{0\}$ with $x\in N_K(x)$.
\end{lemma}

\begin{proof}
By compactness, there exists $x\in K$ such that $\|x\|=\max_{x'\in K}\|x'\|$. Then for any $x'\in K$, the Cauchy-Schwartz inequality yields
\[
\langle x,x'-x\rangle = \langle x,x'\rangle -\|x\|^2 \le \|x\| \|x'\| - \|x\|^2 \le 0.
\]
Thus $x\in N_K(x)$. Finally, $x\ne0$ follows since $n\ge1$ and $0\in\interior K$.
\end{proof}

\begin{lemma} \label{L:K unbounded}
Suppose that $n\ge 0$ and that $K$ is compact with $0\in\interior K$. Let $c:\R^n\to\S^n$ be an affine map parallel to $K$ with $c(0)$ invertible. Then $n=0$.
\end{lemma}

\begin{proof}
We may suppose that $c(0)=\id$. To see this, let $A=c(0)^{1/2}\in\S^n$ and define the set $\widehat K=A^{-1} K$ as well as the map $\widehat c:\R^n\to\S^n$ via $\widehat c(y)=A^{-1}c(A y)A^{-1}$. Then $\widehat K$ is again compact and convex with $0\in\interior\widehat K$. Moreover, by chasing the definitions one verifies that
\[
N_{\widehat K}(y) = A N_K(A y).
\]
Thus if $v\in N_{\widehat K}(y)$, then $v=Au$ for some $u\in N_K(A y)$, whence $\widehat c(y)v=A^{-1}c(A y) u=0$. It follows that $\widehat c$ is parallel to $\widehat K$. In summary, $\widehat K$ and $\widehat c$ satisfy the same properties as $K$ and $c$, and in addition $\widehat c(0)=\id$. We thus assume without loss of generality that
\[
c(x)=\id+\ell(x)
\]
for some linear map $\ell:\R^n\to \S^n$.

Assume for contradiction that $n\ge 1$. By Lemma~\ref{L:x normal at x} there exists some $x\in K\setminus\{0\}$ with $x\in N_K(x)$. Since $c$ is parallel to $K$, one has $0=c(x)x=x+\ell(x)x$. Hence for any $\lambda\in\R$,
\[
c(-\lambda x) x = x - \lambda\ell(x)x = (1+\lambda)x.
\]
By compactness of $K$, there exists $\lambda>0$ such that $-\lambda x\in\partial K$. Thus $N_K(-\lambda x)$ contains some nonzero element $u$, and the set
\[
T = \{ -\lambda x + t\colon \langle u,t\rangle = 0\}
\]
is a supporting hyperplane of $K$ at $-\lambda x$. On the other hand, since $c$ takes values in $\S^n$ and is parallel to $K$,
\[
\langle u,x\rangle = \frac{1}{1+\lambda}\langle u,c(-\lambda x)x\rangle = \frac{1}{1+\lambda}\langle c(-\lambda x)u,x\rangle = 0.
\]
Thus with $t=\lambda x$ we find that $0=-\lambda x+t\in T$, which contradicts the hypothesis that $0\in\interior K$.
\end{proof}

\section{Nonnegative semimartingales}

Variations of the following result are well-known in the literature; see e.g.~\citet[Proposition~3.1]{Spreij/Veerman:2012} or \citet[Lemma~A.1]{filipovic.larsson.16}. We use the convention $Y_{0-}=Y_0$ for the semimartingale $Y$ appearing below.

\begin{lemma}\label{L:nonnegativity}
Let $Y$ be a semimartingale with differential characteristics $(b,c,F)$ with respect to a truncation function $\chi$. Assume $Y\ge0$, $Y_0=0$, and that $b_t$, $c_t$, $\int_\R (1-e^{-y}-\chi(y))F_t(dy)$ are right-continuous at $t=0$. Then $c_0=0$.
\end{lemma}

\begin{proof}
The process $Y'=1-e^{-Y}$ is again a semimartingale, whose differential characteristics $(b',c',F')$ with respect to the truncation function $\chi'(y)=y\bm1_{\{|y|\le1\}}$ can be computed using \citet[Proposition B.1]{kallsen.kruehner.es.15}. One finds that
\begin{align*}
b'_t &= b_t e^{-Y_{t-}} - \frac12 c_t e^{-Y_{t-}} + e^{-Y_{t-}} \int_\R (1-e^{-y}-\chi(y))F_t(dy), \\
c'_t &= c_t e^{-2Y_{t-}}.
\end{align*}
Our assumptions directly imply that
\begin{equation} \label{eqL:nonnegativity1}
\text{$b'_t$ and $c'_t$ are right-continuous at $t=0$.}
\end{equation}
Next, note that $\Delta Y'$ takes values in $[0,1]$, and that $\chi'(y)=y$ for all $y$ in this set. \citet[Lemma~I.4.24 and Proposition~II.2.29]{jacod.shiryaev.03} then imply that $Y'$ is a special semimartingale with canonical decomposition $Y'=N+B'$, where $N$ is a local martingale and $B'$ is given by $B'_t=\int_0^t b'_sds$. Let $N=N^c+N^d$ be the decomposition of $N$ into its continuous and purely discontinuous local martingale parts. Enlarging the probability space if necessary, we can then find a Brownian motion $W$ such that $N^c_t = \int_0^t \sqrt{c'_s}dW_s$. Define
\[
Z = \Ecal(-\phi W)
\]
for some $\Fcal_0$-measurable random variable $\phi\ge0$ to be determined later. Integration by parts yields
\begin{equation} \label{eqL:nonnegativity2}
Z_tY'_t = \int_0^t Y'_{s-}dZ_s + \int_0^t Z_s dY'_s + [Z,Y']_t = M_t + \int_0^t Z_s\left( b'_s - \phi \sqrt{c'_s}\right)ds,
\end{equation}
where $M_t=\int_0^t Y'_{s-}dZ_s+\int_0^t Z_sdN_s$ is a local martingale null at zero. Let $\sigma$ be a strictly positive reducing stopping time for $M$, for instance $\sigma=\inf\{t\ge0\colon M_t\ge 1 \text{ or } Z_t\ge 2\}$. Indeed, $\sigma$ is strictly positive, and since $|\Delta N|=|\Delta Y'|\le1$ we have $|M^\sigma|\le3$. Now, define the stopping time
\[
\tau = \inf\Big\{t\ge 0\colon b'_t \ge 1+b'_0 \text{ or } c'_t \le \frac{c'_0}{4} \text{ or } Z_t \le \frac{1}{2} \Big\} \wedge \sigma \wedge 1.
\]
Due to \eqref{eqL:nonnegativity1} we have $\tau>0$ on $\{c'_0>0\}$. Set $\phi= 2(2+0\vee b'_0)/\sqrt{c'_0}$ on $\{c'_0>0\}$, and $\phi=0$ on $\{c'_0=0\}$. Then, on $\{c'_0>0\}$ and for all $s\in[0,\tau)$ we have $b'_s-\phi\sqrt{c'_s} \le -1$ and $Z_s\ge 1/2$. On $\{c'_0=0\}$ we have $\tau=0$. Therefore, in view of \eqref{eqL:nonnegativity2}, we get
\[
0 \le \E[Z_\tau Y'_\tau ] = \E\left[ \int_0^1 \bm1_{[0,\tau)}(s) Z_s\left( b'_s - \phi \sqrt{c'_s}\right)ds \right] \le -\frac{1}{2} \int_0^1 \P(\tau>s) ds.
\]
Thus $\P(c'_0>0)=\P(\tau>0)=0$, which proves the lemma since $c_0=c'_0$.
\end{proof}


\begin{thebibliography}{16}
\providecommand{\natexlab}[1]{#1}
\providecommand{\url}[1]{\texttt{#1}}
\expandafter\ifx\csname urlstyle\endcsname\relax
  \providecommand{\doi}[1]{doi: #1}\else
  \providecommand{\doi}{doi: \begingroup \urlstyle{rm}\Url}\fi

\bibitem[Carath{\'e}odory(1907)]{caratheodory.07}
C.~Carath{\'e}odory.
\newblock {\"U}ber den variabilit\"atsbereich der koeffizienten von
  potenzreihen, die gegebene werte nicht annehmen.
\newblock \emph{Mathematische Annalen}, 64:\penalty0 95--115, 1907.

\bibitem[Cuchiero(2011)]{cuchiero.11}
C.~Cuchiero.
\newblock \emph{Affine and polynomial processes}.
\newblock PhD thesis, ETH ZURICH, 2011.

\bibitem[Cuchiero et~al.(2011)Cuchiero, Filipovic, Mayerhofer, and
  Teichmann]{cuchiero.al.11}
C.~Cuchiero, D.~Filipovic, D.~Mayerhofer, and J.~Teichmann.
\newblock Affine processes on positive semidefinite matrices.
\newblock \emph{The Annals of Applied Probability}, 21\penalty0 (2):\penalty0
  397--463, 2011.

\bibitem[Cuchiero et~al.(2016)Cuchiero, Keller-Ressel, Mayerhofer, and
  Teichmann]{cuchiero/etal:2016}
Christa Cuchiero, Martin Keller-Ressel, Eberhard Mayerhofer, and Josef
  Teichmann.
\newblock Affine processes on symmetric cones.
\newblock \emph{Journal of Theoretical Probability}, 29\penalty0 (2):\penalty0
  359--422, 2016.

\bibitem[Duffie et~al.(2000)Duffie, Pan, and Singleton]{duffie.al.00}
D.~Duffie, J.~Pan, and K.~Singleton.
\newblock Transform analysis and asset pricing for affine jump-diffusions.
\newblock \emph{Econometrica}, 68:\penalty0 1343--1376, 2000.

\bibitem[Duffie et~al.(2003)Duffie, Filipovic, and Schachermayer]{duffie.al.03}
D.~Duffie, D.~Filipovic, and W.~Schachermayer.
\newblock Affine processes and applications in finance.
\newblock \emph{The Annals of Applied Probability}, 13:\penalty0 984--1053,
  2003.

\bibitem[Duffie and Singleton(1999)]{duffie/singleton:1999}
Darrell Duffie and Kenneth~J Singleton.
\newblock Modeling term structures of defaultable bonds.
\newblock \emph{Review of Financial studies}, 12\penalty0 (4):\penalty0
  687--720, 1999.

\bibitem[Ethier and Kurtz(1986)]{ethier.kurtz.86}
S.~Ethier and T.~Kurtz.
\newblock \emph{Markov {P}rocesses. {C}haracterization and {C}onvergence}.
\newblock Wiley, New York, 1986.

\bibitem[Filipovi\'c(2009)]{Filipovic:2009fk}
D.~Filipovi\'c.
\newblock \emph{Term-structure models: a graduate course}.
\newblock Springer finance. Textbook. Springer, Dordrecht, 2009.
\newblock ISBN 9783540097266 (hardcover : alk. paper).

\bibitem[Filipovi\'c and Larsson(2016)]{filipovic.larsson.16}
D.~Filipovi\'c and M.~Larsson.
\newblock Polynomial diffusions and applications in finance.
\newblock \emph{Finance \& Stochastics}, 20:\penalty0 931--972, October 2016.

\bibitem[Jacod and Shiryaev(2003)]{jacod.shiryaev.03}
J.~Jacod and A.~Shiryaev.
\newblock \emph{Limit {T}heorems for {S}tochastic {P}rocesses}.
\newblock Springer, second edition, 2003.

\bibitem[Kallsen and Kr\"uhner(2015)]{kallsen.kruehner.es.15}
J.~Kallsen and P.~Kr\"uhner.
\newblock On a {H}eath-{J}arrow-{M}orton approach for stock options.
\newblock \emph{Finance \& Stochastics}, 19, 2015.
\newblock Electronical Supplement.

\bibitem[Keller-Ressel and Mayerhofer(2015)]{kellerressel/Mayerhofer:2015}
Martin Keller-Ressel and Eberhard Mayerhofer.
\newblock Exponential moments of affine processes.
\newblock \emph{The Annals of Applied Probability}, 25\penalty0 (2):\penalty0
  714--752, 2015.

\bibitem[Piazzesi(2010)]{piazzesi:2010}
Monika Piazzesi.
\newblock Affine term structure models.
\newblock \emph{Handbook of financial econometrics}, 1:\penalty0 691--766,
  2010.

\bibitem[Rockafellar(1970)]{rockafellar.70}
T.~Rockafellar.
\newblock \emph{Convex Analysis}.
\newblock Princeton University Press, Princeton, 1970.

\bibitem[Spreij and Veerman(2012)]{Spreij/Veerman:2012}
P.~Spreij and E.~Veerman.
\newblock Affine diffusions with non-canonical state space.
\newblock \emph{Stochastic Analysis and Applications}, 30:\penalty0 605--641,
  2012.

\end{thebibliography}

\end{document}